\documentclass[journal]{IEEEtran}

\usepackage{bbm}
\usepackage{mathrsfs}
\usepackage{verbatim}
\usepackage{amsmath}
\usepackage{amsfonts}
\usepackage{bm}
\usepackage{graphicx}
\usepackage{caption}
\usepackage{float}

\usepackage{amsthm}
\usepackage{xcolor}
\usepackage[colorinlistoftodos]{todonotes}
\usepackage{mathtools}
\usepackage{cite}
\usepackage{hyperref}

\theoremstyle{definition}
\newtheorem{assumption}{Assumption}
\newtheorem{definition}{Definition}
\newtheorem{remark}{Remark}

\theoremstyle{plain}

\newtheorem{theorem}{Theorem}
\newtheorem{lemma}{Lemma}

\DeclareMathOperator*{\arginf}{arg\, inf}

\begin{document}

\title{The Prescription Approach to Decentralized Stochastic Control with Word-of-Mouth Communication}

\author{Aditya Dave, {\itshape{Student Member, IEEE,}} and Andreas A. Malikopoulos, {\itshape{Senior Member, IEEE}} 
	\thanks{This research was supported in part by ARPAE's NEXTCAR program under the award number DE-AR0000796 and by the Delaware Energy Institute (DEI).}%
	\thanks{The authors are with the Department of Mechanical Engineering, University of Delaware, Newark, DE 19716 USA (email: \texttt{adidave@udel.edu; andreas@udel.edu).}} }


\maketitle

\thispagestyle{empty} 


\begin{abstract}
In this paper, we analyze a network of agents that communicates through word of mouth. In a word-of-mouth communication system, every agent communicates with her neighbors with delays in communication. This is a non-classical information structure where the topological and temporal restrictions in communication mean that information propagates slowly through the network. We present the prescription approach to derive structural results for problems with word-of-mouth communication. The structural results lead to optimal control strategies in time-invariant spaces. We also present a comparison of our results with the common information approach.
\end{abstract}


\IEEEpeerreviewmaketitle

\section{Introduction}


\IEEEPARstart{T}{he} interdependence of many engineering systems can enable the development of a novel framework to process large amounts of data and deliver real-time control actions that optimize associated benefits. As we move to increasingly complex systems \cite{Malikopoulos2016c}, new decentralized control approaches are needed to optimize the impact on system behavior of the interaction between its entities.
Centralized stochastic control \cite{23} has been a ubiquitous approach to control complex systems \cite{Malikopoulos2015}, \textcolor{black}{and is well understood.} It refers to multi-stage optimization problems with external disturbances and noisy observations controlled by a single decision maker. A key assumption is that the decision maker perfectly recalls all past control actions and observations. The information available to an agent when making a decision is called the \textit{information structure} of the system. If every agent has access to the same information and a perfect recall of the information, the information structure is called the \textit{classical information structure}.
While centralized systems have been extensively studied, their core assumption does not hold for many applications involving multiple agents, e.g.,  connected automated vehicles \cite{Malikopoulos2018}, a swarm of drones \cite{Malikopoulos2019a}, and smart grids \cite{Hiskens2012}. In such systems, all agents simultaneously make a decision based only on their memory and local information received through communication with other agents. Such agent-to-agent communication might be generally  delayed or costly \cite{Malikopoulos2018d}. Thus, it is not possible to compute a centralized estimate of the complete state of the system from the observation history of a single agent. Such multi-stage optimization problems \cite{Malikopoulos2015b} are known as decentralized stochastic control problems.

In this paper, we study a decentralized system with a word-of-mouth information structure.  In the absence of any entity facilitating transmission of information among agents (a shared memory), the communication among agents {occurs through the \textit{word of mouth}.}
In a word-of-mouth communication, we model the system as a network of agents connected by communication links. An agent can directly communicate with her neighbors in the network, {with some associated delays}. Thus, information from each agent propagates through the network by her neighbors who share it with their neighbors, and so on. This problem has a non-classical information structure because of the delays in communication.
{Such a model} can describe many different decentralized systems with limited communication. As an example, consider a team of drones {that} conducts a rescue operation at an inaccessible location. At any time, all drones simultaneously {must make} the best possible decision \textcolor{black}{under asymmetric information}. Due to hardware limitations, a drone can only communicate within a \textcolor{black}{limited} range. The range of the entire team of drones, however, may need to extend far beyond the range of a single drone. In this situation, the drones need to communicate through word of mouth to make optimal decisions.

\subsection{Related Work}

Decentralized stochastic control is fundamentally different from and significantly more challenging than centralized stochastic control. The most common approach in centralized stochastic control is dynamic programming (DP), which is not directly applicable to decentralized systems since no {general} separation results between estimation and control are known \cite{witsenhausen1968counterexample}. Thus, new approaches are needed to address these problems. There are three general approaches in the literature that use techniques from centralized stochastic control.

The person-by-person approach aims to transform the problem into a centralized stochastic control problem from the point of view of a single agent. This is done by arbitrarily fixing the control strategies for all agents except for one agent ${i}$. The control strategy of agent $i$ is then optimized for this new problem, assuming the fixed strategies used by the other agents are known to agent $i$. This allows for the use of techniques from centralized stochastic control, including DP.
This process is then repeated for all agents until a stable equilibrium is attained. This control strategy {is} called a person-by-person optimal strategy. In general, the solution is not globally optimal \cite{2}.
\textcolor{black}{However, the solution's structural properties may guarantee a time-invariant optimal strategy.}
An early application of this idea was found in problems of real-time communication using encoders and decoders \cite{4,5,9,11,13}.
The person-by-person approach has also been used in decentralized hypothesis testing and quickest detection problems \cite{8}.
Some other applications have appeared in networked control systems \cite{6} and team decision problems with partially nested information structures \cite{2,1}. In special cases, such as team decision problems with static information structures \textcolor{black}{and convex cost,} the person-by-person solution is globally optimal \cite{marschak1972economic}.

The designer's approach {considers} the point of view of a designer with knowledge of the model and statistics of all random variables in the system. The designer's task is to choose the optimal control strategy, or \textit{design}, for the system.
This is done by transforming the problem into a centralized planning problem, and then using DP to derive the optimal strategy. The control action for the designer's planning problem is the control law that the designer assigns to each agent at each time step. DP ensures global optimality of these assignments. \textcolor{black}{However, the designer's approach, first introduced by Witsenhausen \cite{3} for general models, suffers from high computational complexity.
This approach was later applied in controls with a fixed memory \cite{10}. This approach was further developed in \cite{mahajan2008sequential}.}

The most recent approach in decentralized control is the \textit{common information approach}, first presented for problems with partial history sharing \cite{17}, \cite{14}, where each agent shares a subset of her past observations and control actions to a shared memory accessible to every agent in the system. Then, each agent makes a decision using information available in the shared memory (the common information) and the recall of her own history.
The solution is derived by reformulating the system from the viewpoint of a fictitious \textit{coordinator} with access only to the shared information, {whose} task is to prescribe control strategies to each agent. This approach has been used to derive \textit{structural results} and a DP decomposition for problems with delayed information to the shared memory \cite{14}. Structural results are properties of optimal control strategies that address how we can compress information that is increasing with time to a sufficient statistic that takes values in a time-invariant space.
The common information approach has also been used in problems with control sharing information structure \cite{16}, stochastic games with asymmetric information \cite{20} and teams with mean-field sharing \cite{Mahajan2015}. There are also some earlier papers that used similar ideas to analyze specific information structures \cite{yuksel2009stochastic, yoshikawa1978decomposition, aicardi1987decentralized}.

\subsection{Contributions}

\textcolor{black}{The common information approach is considered as the standard approach to analyze systems with a word-of-mouth information structure. However, when the maximum delay in communication across agents is large, the common information approach can lead to a large number of computations to derive the optimal strategies.
This has motivated us to develop the prescription approach for systems with different delays in communication among different agents. Some related preliminary results can be found in \cite{Aditya_2019, dave2020structural}.} 

\textcolor{black}{The contributions of this paper are: We present the prescription approach that leads to improved structural results for decentralized systems with a word-of-mouth information structure. These structural results lead to a smaller search space for optimal strategies when compared to the common information approach (see Section V-B for details), or any other approach. The structural properties and derivations of our results are distinct from existing results in literature. In fact, the common information approach can be considered as a special case of the prescription approach. Furthermore, our results are applicable to a wide range of decentralized control problems with delayed, asymmetric information sharing among agents.}

\subsection{A static system}

\textcolor{black}{In this subsection, we present a simple example to illustrate the main idea behind the prescription approach and make a comparison with the common information approach. Consider a static system with three agents. Let $X$ denote the state of the system and $Y^1, \; Y^2, \; Y^3$ be the observations of each agent about $X$. All random variables are binary and their joint probability distribution is known.}
\textcolor{black}{The information available to agents $1$, $2$, and $3$ is $M^1 = \{Y^1, Y^2, Y^3\}$, $M^2 = \{Y^2, Y^3\}$, and $M^3 = \{Y^3\}$, respectively. Every agent $k \in \{1,2,3\}$ selects a control action as $U^k = g^k(M^k)$ and the system incurs a cost $c(X, U^1, U^2, U^3)$. The problem is to find the control strategy $\boldsymbol{g} = (g^1,g^2,g^3)$ that minimizes the performance criterion
    $\mathcal{J}(\boldsymbol{g}) = \mathbb{E}^{\boldsymbol{g}}[c(X, U^1, U^2, U^3)]$.}
\textcolor{black}{Using a brute force search, we would have to compute the performance of $2^8 \times 2^4 \times 2^2 = 16,384$ strategies before arriving at the optimal control strategy.}

\textcolor{black}{The key idea of this paper is to partition the memory of different agents into accessible and inaccessible information to simplify the search for optimal strategies. We formally define these terms in Section III-A. Here, we simply note that the accessible information of agents $1$, $2$, and $3$ are $A^1 = \{Y^1, Y^2, Y^3\}$, $A^2 = \{Y^2, Y^3\}$, and $A^3 = \{Y^3\}$, respectively. The accessible information of agent $3$ is also the common information in the system.
The inaccessible information of agent $1$ with respect to agent $1$ is $L^{[{{1}},{{1}}]} = \emptyset$, with respect to agent $2$ is $L^{[{{1}},{{2}}]} = \{Y^1\}$, and with respect to agent $3$ is $L^{[{{1}},{{3}}]} = \{Y^1, Y^2\}$. Similarly, the inaccessible information of agent $2$ with respect to agent $2$ is $L^{[{{2}},{{2}}]} = \emptyset$, and with respect to agent $3$ is $L^{[{{2}},{{3}}]} = \{Y^2\}$. Finally, the inaccessible information of agent $3$ with respect to agent $3$ is $L^{[{{3}},{{3}}]} = \emptyset$. Note that $M^1 = A^1 \cup L^{[1,1]} = A^2 \cup L^{[1,2]}$ and so on.}

\textcolor{black}{Now we consider that every agent $k \in \{1,2,3\}$ generates prescriptions $(\Gamma^{[k,1]}, \Gamma^{[k,2]}, \Gamma^{[k,3]})$ corresponding to every other agent in $\{1,2,3\}$, respectively, by only observing the accessible information $A^k$. The prescription $\Gamma^{[k,i]}$ is a mapping from the inaccessible information $L^{[i,k]}$ to the control action $U^i$. Thus, for agent $3$, for every $k \in \{1,2,3\}$, we can write $\Gamma^{[{{3}},{{k}}]} = \psi^{[{{3}},{{k}}]}(A^3) = \psi^{[{{3}},{{k}}]}(Y^3),$ where $\psi^{[{{3}},{{k}}]}$ is the \textit{prescription law} for the prescription of agent $3$ for agent $k$. 
Let $\Gamma^{[{{3}},{{k}}]}_n = \psi^{[{{3}},{{k}}]}(Y^3 = n)$ for $n \in \{0,1\}$. Then, the performance criterion for agent $3$ is given by
\begin{align}
    \mathcal{J}(\boldsymbol{\psi}^{3}) = \mathbb{E}^{\boldsymbol{\psi}^{3}}[c(X, U^1, U^2, U^3)] & \nonumber \\
    = \sum_{n \in \{0,1\}} \mathbb{P}(Y^3 = n) \cdot \mathbb{E}^{\boldsymbol{\psi}^{3}}\Big[c\big(&X, \Gamma^{[{{3}},{{1}}]}_n(Y^{{1}}, Y^{{2}}), \nonumber \\
    &\Gamma^{[{{3}},{{2}}]}_n(Y^{{2}}), \Gamma^{[{{3}},{{3}}]}_n \big) | Y^3 = n\Big], \nonumber
\end{align}
where $\boldsymbol{\psi}^{3} = (\psi^{[{{3}},{{1}}]}, \psi^{[{{3}},{{2}}]}, \psi^{[{{3}},{{3}}]})$.
Minimizing the performance criterion above is equivalent to minimizing the original expected cost (see Lemma \ref{lem_equivalence}). In this new cost, we can minimize the two terms separately, leading to a computation of $(2^4 \times 2^2 \times 2) + (2^4 \times 2^2 \times 2) = 256$ strategies. Note that the number prescriptions obtained for agent $3$ are the same as those obtained using the common information approach in \cite{17}. Thus, the common information approach also leads to a computation of $256$ strategies.}

\textcolor{black}{
Now, consider the prescriptions for agent $2$. For $k \in \{1,2\}$, we have $\Gamma^{[{{2}},{{k}}]} = \psi^{[{{2}},{{k}}]}(A^2) = \psi^{[{{2}},{{k}}]}(Y^2, Y^3)$, where $\psi^{[{{2}},{{k}}]}$ is the prescription law for the prescription of agent $2$ for agent $k$. For the case $k=3$, $\Gamma^{[{{2}},{{3}}]} = \Gamma^{[{{3}},{{3}}]}$ (see Lemmas \ref{lem_psi_relation1}-\ref{lem_psi_relation2}). Let $\Gamma^{[{{3}},{{k}}]}_{m,n} = \psi^{[{{3}},{{k}}]}(Y^2 = m, Y^3 = n)$, $n,m \in \{0,1\}$ and $k \in \{1,2\}$. The performance criterion for agent $2$ is
\begin{multline*}
    \mathcal{J}(\boldsymbol{\psi}^{2}) = \mathbb{E}^{\boldsymbol{\psi}^{2}}[c(X, U^1, U^2, U^3)]  \\
    = \sum_{m, n \in \{0,1\}} \mathbb{P}(Y^2 = m, Y^3 = n)  \\
    \cdot \mathbb{E}^{\boldsymbol{\psi}^{2}}\Big[c\big(X, \Gamma^{[{{2}},{{1}}]}_{m,n}(Y^{{1}}), \Gamma^{[{{2}},{{2}}]}_{m,n}, \Gamma^{[{{3}},{{3}}]}_n\big) | Y^2 = m, Y^3 = n\Big],
\end{multline*}
where $\boldsymbol{\psi}^{2} = (\psi^{[{{2}},{{1}}]}, \psi^{[{{2}},{{2}}]}, \psi^{[{{2}},{{3}}]})$. We show in Lemma 5 that the optimal cost achieved by agent $2$ is the same as the one of agent $3$.
However, in the cost for agent $2$, we can minimize the four terms separately, having to compute only $(2^2 \times 2 \times 2) \times 4 = 64$ strategies. Using a similar procedure, we can show that the optimal prescription strategy for agent $1$ can also be obtained with $64$ computations. Note that the cost for agents $1$, $2$, and $3$ are different, but all lead to the same optimal cost.
The costs for agents $1$ and $2$ lead to an improvement in computation time of optimal strategies when compared to the common information approach. This example shows that even in static systems, there are avenues beyond the common information approach to improve the computation time for optimal strategies.}
\textcolor{black}{\begin{remark}
Consider a larger set of agents $\mathcal{K} = \{1, \dots, K\}$, with the same information structure $M^{k+1} \subseteq M^k$, $k \in \mathcal{K} \setminus \{K\}$. The agent $K$ must still compute the same number of strategies as in the common information approach. However, the number of computations decreases with the index of the agent from $K$ all the way to $1$, implying that the fastest computation of optimal strategies is achieved by agent $1$.
\end{remark}}

\subsection{Organization}

The remainder of the paper is organized as follows: In Section II, we present the model of the decentralized stochastic problem with a word-of-mouth information structure. In section III, we introduce the prescription approach along with the associated properties of the prescription functions. In Section IV, we use the prescription approach to derive structural results from the point of view of various agents in the system. \textcolor{black}{In Section V, we simplify the structural results and compare them with the common information approach.}
Finally, we conclude with some observations and future research directions in Section VI.

\section{Problem Formulation}

\subsection{Network Description}

Consider a directed network of $K \in \mathbb{N}$ agents represented by a strongly connected graph $\mathcal{G} =(\mathcal{K}, \mathcal{E})$, where $\mathcal{K} := \{1,\dots,K\}$ is the set of agents and $\mathcal{E}$ is the set of links. A direct communication link from any agent $k \in \mathcal{K}$ to any agent $j \in \mathcal{K}$ is denoted by $(k,j) \in \mathcal{E}$, which is characterized by a delay of $\delta^{[{{k}},{{j}}]} \in \mathbb{N}$ time steps for transferring information from $k$ to $j$.

When agent $k$ sends out information to some agent $j$ in her neighborhood, we call the act \textit{transmission of information}. The information transmitted by agent $k$ at time $t \in \mathbb{N}$ is received by agent $j$ at time $t + \delta^{[{{k}},{{j}}]}$. For any agent $k$, the acts of \textcolor{black}{reception} and transmission of information occur at different instances within every time step, as discussed in Section II-D.

\begin{definition}
\label{path}
\textcolor{black}{For two agents $k, j \in \mathcal{K}$, a \textit{path}  $q^{[{{k}},{{j}}]}_a$, $a \in \mathbb{N}$, from $k$ to $j$ is the sequence $\{k_n\}_{n=1}^m$, $m \in \mathcal{K}$, such that: (1) $k_1 = k$ and $k_m = j$, (2) $k_n \in \mathcal{K}$ for $n \in \mathcal{K}$, and (3) there exists a link $(k_{n-1},k_n) \in \mathcal{E}$ for $n \in \mathcal{K}\setminus\{1\}$.}
\end{definition}

\textcolor{black}{Let $\mathcal{Q}^{[{{k}},{{j}}]} = \{q^{[{{k}},{{j}}]}_a: a \in \mathbb{N}\}$ be the set with all paths from agent $k$ to agent $j$.}

\begin{definition}
\label{path_delay}
Consider agents $k,j \in \mathcal{K}$ with a path $q^{[{{k}},{{j}}]}_a \in \mathcal{Q}^{[{{k}},{{j}}]}$. The \textit{communication delay} $d^{[{{k}},{{j}}]}_a \in \mathbb{N}$ for $q^{[{{k}},{{j}}]}_a$ is {$d^{[{{k}},{{j}}]}_a = \delta^{[k,k_2]}+\dots+\delta^{[k_{m-1},j]},$}
where $\delta^{[k_{n-1},k_n]}$ is the delay in information transfer through the link $(k_{n-1},k_n) \in \mathcal{E}$.
\end{definition}

The \textit{information path} from agent $k$ to agent $j$ in the network is the path with the least possible delay.

\begin{definition}
The \textit{information path} from $k$ to $j$, denoted by $(k \rightarrow j)$, is the path $q^{[{{k}},{{j}}]} \in \mathcal{Q}^{[{{k}},{{j}}]}$ such that {$d^{[{{k}},{{j}}]} = \min\left\{d_1^{[{{k}},{{j}}]},\dots,d_b^{[{{k}},{{j}}]}\right\},$}
where $b := |\mathcal{Q}^{[{{k}},{{j}}]}|$.
\end{definition}

The property of strong connectivity in our network ensures that there is always an information path $(k \rightarrow j)$ from every agent $k \in \mathcal{K}$ to every agent $j \in \mathcal{K}$. \textcolor{black}{When there are multiple paths with the same least delay, any of these paths can be designated to be the information path.} We denote the delay associated with the information path $(k \rightarrow j)$  by $d^{[{{k}},{{j}}]}$ and, by convention, let $d^{[k,k]} = 0$. Since the links in the network are directed, the delay $d^{[{{k}},{{j}}]}$ in communication from $k$ to $j$ is not necessarily the same as the delay $d^{[j,k]}$ from $j$ to $k$.

\subsection{System Description}
The network of agents is considered a discrete-time system that evolves up to a finite time horizon $T \in \mathbb{N}$. At any time $t = 0,1,\dots,T$, the state of the system $X_t$ takes values in a finite set $\mathcal{X}$ and the control variable $U_t^k$ generated by agent $k \in \mathcal{K}$ takes values in a finite set $\mathcal{U}^k$. Let ${U}_t^{1:K}$ denote the vector $(U_t^1,\dots,U_t^K)$. Starting at the initial state $X_0$, the evolution of the system follows the state equation
\begin{equation}
X_{t+1}=f_t\left(X_t,U_t^{1:K},W_t\right), \label{st_eq}
\end{equation}
where $W_t$ is the external disturbance to the system represented as a random variable taking values in a finite set $\mathcal{W}$. At time $t$ every agent $k$ makes an observation $Y_t^k$, 
\begin{equation}
Y_t^k=h_t^k(X_t,V_t^k), \label{ob_eq}
\end{equation}
taking values in a finite set $\mathcal{Y}^k$ through a noisy sensor, where $V_t^k$ takes values in the finite set $\mathcal{V}^k$ and represents the noise in measurement.
Agent $k$ selects a control action $U_t^k$ from the set of feasible control actions $\mathcal{U}^k$ as a function of her information structure. The information structure is different for each agent $k \in \mathcal{K}$ because of the means of communication and topology of the network. We discuss the information structure in detail in Section II-E. After each agent $k$ generates a control action $U_t^k$, the system incurs a cost $c_t(X_t,U_t^1,\dots,U_t^K) \in \mathbb{R}$.

\subsection{Assumptions}

In our modeling framework above, we impose the following assumptions:

\begin{assumption}
\label{top_asu}
The network topology is strongly connected with self-loops, known a priori, and does not change with time.
\end{assumption}

With a known and invariable network topology, every agent can keep track of what information is accessible to other agents in the network.

\begin{assumption}
\label{prim_asu}
The external disturbances $\{W_t: t=0,\dots,T\}$ and noises in measurement $\{V_t^k: t=0,\dots,T; k=1,\dots,K\}$ are sequences of independent random variables that are also independent of each other and of the initial state $X_0$. 
\end{assumption}

The external disturbances, noises in measurement, and initial state are referred to as the \textit{primitive} random variables, and their probability distributions are known to all decision makers. This allows us to analyse the system as a controlled Markov chain.

\begin{assumption}
\label{func_asu}
The state functions $(f_{t}: t=0,\dots,T)$, observation functions $(h_{t}^{k}: t=0,\dots,T; k=1,\dots,K)$, the cost functions $(c_{t}: t=0,\dots,T),$ and the set of all feasible control strategies $G$ are known to all agents.
\end{assumption}

These functions and the set of feasible control strategies (defined in Section II-E) form the basis of the decision making problem.

\begin{assumption}
Each agent perfectly recalls all information that enters her memory.
\end{assumption}

Perfect recall of the data from the memory of every agent is an essential assumption for the structural results derived in this paper.

\begin{figure}[ht!]
  \centering
  \captionsetup{justification=centering}
  \includegraphics[width=0.9\linewidth, keepaspectratio]{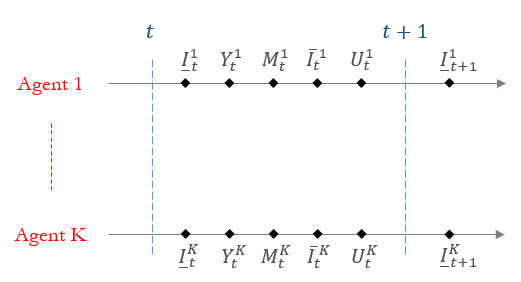}
  \caption{Sequence of activities}
  \label{fig:update}
\end{figure}

We summarize below the sequence of activities taken by agent $k \in \mathcal{K}$ at time $t$ (see Fig. \ref{fig:update}):
\begin{enumerate}
\item The state $X_t$ is updated based on (\ref{st_eq}).
\item Agent $k$ receives information from all agents in $\mathcal{K}$, denoted by $\underline{I}_t^k$.
\item Agent $k$ makes an observation about the state $Y_t^k$ based on (\ref{ob_eq}).
\item Agent $k$ updates her memory, $M_t^k$, defined in the next section, on a given protocol.
\item Agent $k$ transmits information, $\overline{I}_t^k$, to every agent $j \in \mathcal{K},$ which is directly connected with $k$.
\item Agent $k$ generates a control action $U_t^k$.
\end{enumerate}

\subsection{Information Structure of the System}

The information structure of the system is characterized by the network topology and delays along information paths described in Section I-B. In the word-of-mouth information structure, every agent $j \in \mathcal{K}$ at time $t$ transmits the information $\overline{I}_t^j := \{Y_{t}^j,U_{t-1}^j\}$ to every other agent in the network through relevant information paths. Agent $k \in \mathcal{K}$ receives information $\overline{I}_t^j$ at time $t + d^{[j,k]}$, where $d^{[j,k]}$ is the communication delay from $j$ to $k$.
Then, the information available to agent $k$ at time $t$ is the collection of information she received from every agent $j \in \mathcal{K}$ at time steps $0$ through $t$.

\begin{definition}
The \textit{memory} of agent $k \in \mathcal{K}$ is the set of random variables $M_t^k$ that takes values in the finite collection of sets $\mathcal{M}_t^k$, and is given by
\begin{align}
M_t^k := &\left\{Y^j_{0:t-d^{[j,k]}},U_{0:t-d^{[j,k]}-1}^j: j \in \mathcal{K} \right\}, \label{eq_mem}
\end{align}
where $d^{[j,k]}$ is the delay in information transfer from every agent $j \in \mathcal{K}$ to agent $k$.
\end{definition}

At time $t$, agent $k$ accesses her memory $M_t^k$ to generate a control action, namely,
\begin{gather}
    U_t^k := g_t^k(M_t^k), \label{u_basic}
\end{gather}
where $g_t^k$ is the control law of agent $k$ at time $t$. We denote the control strategy for each agent by $\boldsymbol{g}^k := (g_0^k,\dots,g_T^k)$ and the control strategy of the system by $\boldsymbol{g} := (\boldsymbol{g}^1,\dots,\boldsymbol{g}^K)$. The set of all feasible control strategies is denoted by $G$.
The performance criterion for the system is given by the total expected cost:
\begin{equation}
\textbf{Problem 1:}~~~~ \mathcal{J}(\boldsymbol{g}) = \mathbb{E}^{\boldsymbol{g}}\left[\sum_{t=0}^T{c_t(X_t,U_t^{1},\dots,U_t^K)}\right], \label{per_cri}
\end{equation}
where the expectation is with respect to the joint probability distribution on all random variables.
The optimization problem is to select the optimal control strategy $\boldsymbol{g}^* \in G$ that minimizes the performance criterion in \eqref{per_cri}, given the probability distributions of the primitive random variables $\{X_0,W_{0:T},V_{0:T}^1,\dots,V_{0:T}^K\}$, and functions $\left\{c_t,f_t,h_t^{1:K}:t=0,\dots,T\right\}$.

\section{\textcolor{black}{The Prescription Approach}}

In this section, we present the prescription approach to derive structural results for our model. \textcolor{black}{In our exposition, we refer only to the memory $M_t^k$ of agent $k \in \mathcal{K}$.}
We begin by defining the set of agents indexed beyond ${{k}}$ in the network \textcolor{black}{to simplify notation.} 

\begin{definition}\label{def:setB}
\textcolor{black}{For agent $k\in\mathcal{K}$, the set of agents indexed beyond ${{k}}$ is $\mathcal{B}^{{{k}}} := \{{{i}} \in \mathcal{K}: {{i}} \geq {{k}}\}$.}
\end{definition}

\subsection{Construction of Prescriptions}

For an agent $k \in \mathcal{K}$, we consider a scenario where the control action $U_t^{{k}}$ is generated in two stages:

(1) Agent $k$ generates a function based on information which is a subset of her memory $M_t^{{k}}$.

(2) This function takes as an input the \textcolor{black}{complement} of the subset in (1), and yields the control action $U_t^{{k}}$.

We call these functions \textit{prescriptions}. These functions allow us to construct an optimization problem of selecting the optimal \textit{prescription strategy} that is equivalent to the problem of selecting the optimal control strategy $\boldsymbol{g}^{*k}$, as shown in Section III-C. In this subsection, we construct the subset of the memory $M_t^{{k}}$ and prescriptions for \textcolor{black}{every} agent ${{k}}$ without changing the information structure of the system.

\begin{definition} \label{def_A}
Let $M_t^{{k}}$ be the memory of an agent $k \in \mathcal{K}$ at time $t$. The \textit{accessible information} of agent ${{k}}$ is the set of random variables $A_t^{{k}}$, that takes values in a finite collection of sets $\mathcal{A}_t^{{k}}$, such that
\begin{gather}
    A_t^{{k}} = \bigcap_{{{j}}=1}^{{k}}\left(M_t^{{j}}\right). \label{ainfo_def}
\end{gather}
\end{definition}
For example, we can write \eqref{ainfo_def} for agent ${{k}}=1$ as $A_t^{{1}} = M_t^{{1}}$, and for agent ${{k}}=2$ as $A_t^{{2}} = M_t^{{1}} \cap M_t^{{2}}.$
Based on Definition \ref{def_A}, the accessible information $A_t^{{k}}$ has \textcolor{black}{the following properties:
\begin{align}
    & A_{t-1}^{{k}} \subseteq A_t^{{k}}, \label{ainfo_prop_2} \\
    & A_t^{{i}} \subseteq A_t^{{k}}, \quad \forall {{i}} \in \mathcal{B}^{{k}}, \label{ainfo_prop_1}
\end{align}}
where $\mathcal{B}^{{k}}$ is the set of agents beyond ${{k}}$ (Definition \ref{def:setB}).

\begin{definition}
The \textit{new information} for agent $k$ at time $t$ is the set of random variables $Z_t^{{k}}$ that takes values in a finite collection of sets $\mathcal{Z}_t^{{k}}$, such that
$Z_t^{{k}} := A_t^{{k}} \backslash A^{{k}}_{t-1}.$
\end{definition}

Note that in \eqref{ainfo_prop_1} that the accessible information $A_t^{{i}}$ of an agent ${{i}} \in \mathcal{B}^{{k}}$ is a subset of the memory $M_t^{{k}}$ of agent ${{k}}$.

\begin{definition}
The \textit{inaccessible information} of agent ${{k}} \in \mathcal{K}$   with respect to accessible information $A_t^{{i}}$, ${{i}} \in \mathcal{B}^{{k}}$, is the set of random variables $L_t^{[{{k}},{{i}}]}$ that takes values in a finite collection of sets $\mathcal{L}_t^{[{{k}},{{i}}]}$, such that
\begin{gather}
    L_t^{[{{k}},{{i}}]} := M_t^{{k}} \setminus A_t^{{i}}. \label{inacc_def}
\end{gather}
\end{definition}

Note that the inaccessible information of agent ${{k}}$ with respect to her own accessible information $A_t^{{k}}$ is given by {$L_t^{[{{k}},{{k}}]} := M_t^{{k}} \setminus A_t^{{k}}.$  Furthermore,} the pair of sets $A_t^{{i}}$ and $L_t^{[{{k}},{{i}}]}$ forms a partition of the set $M_t^{{k}}$, such that
\begin{gather}
    M_t^{{k}} = \{L_t^{[{{k}},{{i}}]},A_t^{{i}}\}, \quad \forall {{i}} \in \mathcal{B}^{{k}}. \label{eq_partition}
\end{gather}
We use the partitions of the memory through accessible and inaccessible information to define the \textit{prescription function}.

\begin{definition}
The \textit{prescription function} $\Gamma_t^{[{{k}},{{i}}]}$ of agent ${{k}}$ for an agent ${{i}} \in \mathcal{K}$ is the mapping
\begin{align} \label{pres_func_def}
    \Gamma_t^{[{{k}},{{i}}]} :
    \begin{cases}
        \mathcal{L}_t^{[{{i}},{{k}}]} \to \mathcal{U}^{{i}}, \quad  \text{if } {{i}} \not\in \mathcal{B}^{{k}}, \\
        \mathcal{L}_t^{[{{i}},{{i}}]} \to \mathcal{U}^{{i}}, \quad  \text{if } {{i}} \in \mathcal{B}^{{k}},
    \end{cases}
\end{align}
that takes values in a set of feasible prescription functions $\mathscr{G}_t^{[{{k}},{{i}}]}$.
\end{definition}

\begin{remark}
In Definition 8, the inaccessible information of an agent $k$ is defined with respect to agent ${{i}} \in \mathcal{B}^{{k}}$. Note that in the first part of \eqref{pres_func_def}, we have ${{i}} \not\in \mathcal{B}^{{k}}$, which implies that ${{k}} \in \mathcal{B}^{{i}}$ instead.
\end{remark}

The prescription function of agent ${{k}}$ for itself at time $t$ is $\Gamma_t^{[{{k}},{{k}}]} :\mathcal{L}_t^{[{{k}},{{k}}]} \to \mathcal{U}^{{k}}$. 
Every prescription function $\Gamma^{[{{k}},{{i}}]}_t$ is generated as follows
\begin{align} \label{eq_gen_pres}
    \Gamma_t^{[{{k}},{{i}}]} :=
    \begin{cases}
    {\psi}_t^{[{{k}},{{i}}]}(A_t^{{k}}), \quad \text{if } {{i}} \not\in \mathcal{B}^{{k}}, \\
    {\psi}_t^{[{{k}},{{i}}]}(A_t^{{i}}), \quad \text{if } {{i}} \in \mathcal{B}^{{k}},
    \end{cases}
\end{align}
where we call ${\psi}_t^{[{{k}},{{i}}]}$ the \textit{prescription law} of agent ${{k}}$ for agent ${{i}}$.
We call $\boldsymbol{\psi}^{{k}} := ({\psi}^{[{{k}},{{1}}]}_{0:t},\dots,{\psi}^{[{{k}},{{K}}]}_{0:t})$  the \textit{prescription strategy} of agent ${{k}}$. The set of feasible prescription strategies for agent ${{k}}$ is denoted by $\Psi^{{k}}$. Note that the prescription of agent $k$ for itself is $\Gamma_t^{[{{k}},{{k}}]} = \psi_t^{[{{k}},{{k}}]}(A_t^k).$

\begin{remark}
The prescription $\Gamma_t^{{[{{k}},{{i}}]}}$ of agent $k$ for agent $i$ is only available to agent $k$. The equivalent prescription available to agent $i$ is $\Gamma_t^{{[{{i}},{{i}}]}}$. The relationship between the two is given in Lemmas \ref{lem_psi_relation1} and \ref{lem_psi_relation2}.
\end{remark}

\begin{remark}
Every agent needs to generate prescriptions corresponding to every other agent in the system in order to utilize the structural results presented in this paper. This is highlighted in Section III-D when we introduce the \textit{information state}.
\end{remark}

Next, we define the \textit{complete prescription} of an agent ${{k}}$ to simplify the notation in Section IV.

\begin{definition}
The \textit{complete prescription} for an agent ${{k}} \in \mathcal{K}$   is given by the function
\begin{align}
    \Theta_t^{{k}} : \;
    \mathcal{L}_t^{[{{1}},{{k}}]} \times \dots \times \mathcal{L}_t^{[{{k}},{{k}}]} &\times \mathcal{L}_t^{[{{k+1}},{{k+1}}]} \times \dots \times \mathcal{L}_t^{[{{K}},{{K}}]} \nonumber \\
    &\longrightarrow \mathcal{U}^{{1}} \times \dots \times \mathcal{U}^{{K}},
\end{align}
which takes values in the set of functions $\mathscr{G}_t^{{k}}$.
\end{definition}

\textcolor{black}{Note the complete prescription for agent $k = {{1}}$ is $\Theta_t^{{1}} : \mathcal{L}_t^{[{{1}},{{1}}]} \times \dots \times \mathcal{L}_t^{[{{K}},{{K}}]} \to \mathcal{U}^{{1}} \times \dots \times \mathcal{U}^{{K}}.$}
The complete prescription for agent ${{k}}$  is constructed as $\Theta_t^{{k}} = (\Gamma_t^{[{{k}},{{1}}]},\dots,\Gamma_t^{[{{k}},{{K}}]})$.

\color{black}

\begin{remark}
The accessible information $A_t^{{k}}$ and inaccessible information $L_t^{[{{k}},{{i}}]}$ of an agent ${{k}} \in \mathcal{K}$ with respect to an agent ${{i}} \in \mathcal{B}^{{k}}$, depend on the respective indices of those agents. However, we do not restrict our attention to any special indexing pattern, and thus, our results in Section IV hold for all possible assignments of indices. Furthermore, our problem describes a sequential system, in accordance with the definition in \cite{witsenhausen1971information}. Thus, we can apply Witsenhausen's result on ordering agents \cite{witsenhausen1971information} to state that the optimal control strategies are independent of the indexing pattern of agents.
\end{remark}
\subsection{Illustrative Example}
To illustrate the partitioning of the memory, we present a simple example with three agents. The state of the system at time $t$ is a three-dimensional vector $X_t = (X_t^1,X_t^2,X_t^3)$, where each $X_t^k$ is the state of a local subsystem of agent $k \in \{1,2,3\}$. The state evolves according to
\begin{gather}
    X_{t+1} = f_t(X_t,U_t^1,U_t^2,U_t^3,W_t), \label{example_1}
\end{gather}
where $\{W_t:t=0,\dots,T\}$ denotes the set of mutually independent disturbances that are also independent from the initial state $X_0$. The observation of agent $k$ at time $t$ is noise-free, i.e., $Y_t^k = X_t^k$. The communication delays among agents are given by
$
    d^{[1,2]} = d^{[2,1]} = 
    d^{[1,3]} = d^{[3,1]} = 1,$ and $
    d^{[2,3]} = d^{[3,2]} = 2.
$
Thus, the memories of the three agents at time $t$ are given by
\begin{gather*}
    M_t^1 = \{X^1_{0:t},U^1_{0:t-1},X^2_{0:t-1},U^2_{0:t-2},X^3_{0:t-1},U^3_{0:t-2}\}, \\
    M_t^2 = \{X^1_{0:t-1},U^1_{0:t-2},X^2_{0:t},U^2_{0:t-1},X^3_{0:t-2},U^3_{0:t-3}\}, \\
    M_t^3 = \{X^1_{0:t-1},U^1_{0:t-2},X^2_{0:t-2},U^2_{0:t-3},X^3_{0:t},U^3_{0:t-1}\}.
\end{gather*}
Then, the accessible information of the three agents is given by,
\begin{align*}
    &A_t^{1} = \{X^1_{0:t},U^1_{0:t-1},X^2_{0:t-1},U^2_{0:t-2},X^3_{0:t-1},U^3_{0:t-2}\}, \\
    &A_t^{{2}} = \{X^1_{0:t-1},U^1_{0:t-2},X^2_{0:t-1},U^2_{0:t-2},X^3_{0:t-2},U^3_{0:t-3}\}, \\
    &A_t^{{3}} = \{X^1_{0:t-1},U^1_{0:t-2},X^2_{0:t-2},U^2_{0:t-3},X^3_{0:t-2},U^3_{0:t-3}\}.
\end{align*}
The inaccessible information of agent $1$ is $L_t^{{[1,1]}} = \emptyset,$ $L_t^{{[1,2]}} = \{X^1_t,U^1_{t-1},X^3_{t-1},U^3_{t-2}\}$, and $L_t^{{[1,3]}} = \{X^1_t,U^1_{t-1},X^2_{t-1},U^2_{t-2},X^3_{t-1},U^3_{t-2}\}$.
Similarly, the inaccessible information of agent $2$ is $L_t^{{[2,2]}} = \{X^2_t,U^2_{t-1}\}$ and $L_t^{{[2,3]}} = \{X^2_{t-1:t},U^2_{t-2:t-1}\}$,
and the accessible information for agent $3$ is $L_t^{{[3,3]}} = \{X^3_{t-1:t},U^3_{t-2:t-1}\}$.
The memories, accessible information and inaccessible information of the three agents are illustrated in Fig. \ref{fig:inaccessible} (using rectangles to represent sets). We observe that at any time $t$, for agent $1$, the system satisfies the properties:
\begin{gather}
    A_t^{3} \subset A_t^{2} \subset A_t^{1}, \nonumber \\
    M_t^{1} = \{A_t^{2},L_t^{[{{1}},{{2}}]}\} = \{A_t^{3},L_t^{[{{1}},{{3}}]}\}. \nonumber
\end{gather}
We can derive similar relationships for agents $2$ and $3$.

\color{black}

\begin{figure}[h!]
  \centering
  \captionsetup{justification=centering}
  \includegraphics[height = 8cm, keepaspectratio]{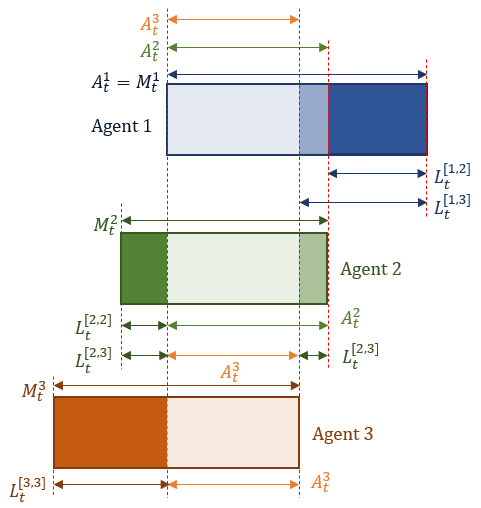}
  \caption{Memory partitions of three agents.}
  \label{fig:inaccessible}
\end{figure}

\subsection{Relationships Between Prescriptions and Control Laws}

In this subsection, Lemmas \ref{lem_psi_g_relation} and \ref{lem_psi_g_relation_inv} imply that every control action $U_t^{{k}}$ of an agent ${{k}} \in \mathcal{K}$ generated through a control strategy $\boldsymbol{g}$ can also be generated through an appropriate prescription strategy $\boldsymbol{\psi}^{{k}}$ and vice versa.

\begin{lemma} \label{lem_psi_g_relation}
For any given control strategy $\boldsymbol{g} \in G$, there exists a prescription strategy $\boldsymbol{\psi}^{{k}} \in \Psi^{{k}}$ such that 
\begin{gather}
    U_t^{{k}} = \Gamma_t^{[{{k}},{{k}}]}\left(L_t^{[{{k}},{{k}}]}\right) = g_t^k(M_t^k). \label{u_presc}
\end{gather}
\end{lemma}

\begin{proof}
Let $A_t^{{k}}$ and $L_t^{[{{k}},{{k}}]}$ be the accessible and inaccessible information of agent $k$, respectively. For any control law ${g}_t^k : \mathcal{M}_t^k \mapsto \mathcal{U}^k$ that generates $U_t^k$ from \eqref{u_basic}, we can select a prescription law $\psi_t^{{k}} : \mathcal{A}_t^{{{k}}} \to \mathscr{G}_t^{[{{k}},{{k}}]}$ such that 
\begin{equation}
\Gamma_t^{[{{k}},{{k}}]}(\cdot) = \psi_t^{{k}}(A_t^{{{k}}})(\cdot).
\end{equation}
Then, the control action is
{$U_t^k = \Gamma_t^{[{{k}},{{k}}]}(L^{[{{k}},{{k}}]}_t) 
    =g_t^{{k}}(A^{{k}}_t,L^{[{{k}},{{k}}]}_t) = g_t^{{k}}(M_t^{{k}}).$}
\end{proof}


\begin{lemma}\label{lem_psi_g_relation_inv}
For any given prescription strategy $\boldsymbol{\psi}^{{k}} \in \Psi^{{k}}$, there exists a control strategy $\boldsymbol{g} \in G$ such that
\begin{gather}
     U_t^{{k}} = g_t^{{k}}(M_t^{{k}}) = \Gamma_t^{[{{k}},{{k}}]}(L_t^{[{{k}},{{k}}]}). \label{u_presc_inv}
\end{gather}
\end{lemma}

\begin{proof}
For any prescription strategy $\boldsymbol{\psi}^{{k}},$ we can construct a control strategy $\boldsymbol{g}^{{k}}$ such that
{$U_t^{{k}} = g_t^{{k}}(M^{{k}}_t) 
= g_t^{{k}}(A^{{k}}_t,L^{[{{k}},{{k}}]}_t) = {\psi}_t^{[{{k}},{{k}}]}(A_t^{{k}})(L^{[{{k}},{{k}}]}_t).$}
\end{proof}

\begin{definition}
A \textit{positional relationship} from agent ${{k}} \in \mathcal{K}$ to agent ${{i}} \in \mathcal{K}$ is given by the function
    \textcolor{black}{$e^{[{{i}},{{k}}]}: \Psi^{{k}} \to \Psi^{{i}}.$}
\end{definition}

Next, we show the existence of a positional relationship $e^{[{{i}},{{k}}]}$ from any agent ${{k}} \in \mathcal{K}$ to any agent ${{i}} \in \mathcal{K}$ with desirable properties that allow us to construct optimal control strategies of all agents in $\mathcal{K}$ from the optimal prescription strategy of just one agent. The following result establishes that using a positional relationship  $e^{[{{i}},{{k}}]}=(e_1^{[{{i}},{{k}}]},\dots,e_T^{[{{i}},{{k}}]})$, agent ${{i}}$ can derive the prescription law for agent ${{j}} \in \mathcal{K}$ given the prescription law of agent ${{k}}$ for agent ${{j}}$, namely
\begin{gather}
   {\psi}_t^{[{{i}},{{j}}]} := e_t^{[{{i}},{{k}}]}\Big({\psi}_t^{[{{k}},{{j}}]}\Big), \quad \forall {{j}} \in \mathcal{K}. \label{eq_e}
\end{gather}

\begin{lemma} \label{lem_psi_relation1}
For any given prescription strategy $\boldsymbol{\psi}^{{k}}$ of agent ${{k}} \in \mathcal{K}$, there exists a positional relationship $e^{[{{i}},{{k}}]}$, ${{i}} \in \mathcal{B}^{{k}}$, such that a prescription strategy $\boldsymbol{\psi}^{{i}}$ of agent ${{i}}$ generated from \eqref{eq_e} yields:
\begin{align}
   \text{\emph{1. }} {\Gamma}_t^{[{{i}},{{j}}]}(L_t^{[{{j}},{{j}}]}) &= {\Gamma}_t^{[{{k}},{{j}}]}(L_t^{[{{j}},{{j}}]}),   \quad \text{if } {{j}} \in \mathcal{B}^{{i}}, \nonumber \\
   \text{\emph{2. }} {\Gamma}_t^{[{{i}},{{j}}]}(L_t^{[{{j}},{{i}}]}) &= {\Gamma}_t^{[{{k}},{{j}}]}(L_t^{[{{j}},{{j}}]}),   \quad \text{if } {{j}} \in \mathcal{B}^{{k}}, {{j}} \not\in \mathcal{B}^{{i}}, \nonumber \\
   \text{\emph{3. }} {\Gamma}_t^{[{{i}},{{j}}]}(L_t^{[{{j}},{{i}}]}) &= {\Gamma}_t^{[{{k}},{{j}}]}(L_t^{[{{j}},{{k}}]}),   \quad \text{if } {{j}} \not\in \mathcal{B}^{{k}}.
    \label{lem_3_condition}
\end{align}
\end{lemma}

\begin{proof}
Let $g_t^{{j}}$ denote the control law of an agent $j \in \mathcal{K}$ at time $t$. To prove the result, we construct the control law $g_t^{{j}}$ and the prescription law ${\psi}_t^{[{{i}},{{j}}]}$ for three cases, given a prescription strategy $\boldsymbol{\psi}^{{k}}$.

1) If ${{j}} \in \mathcal{B}^{{i}}$, the control law $g_t^{{j}}$ can be constructed from the prescription law {${\psi}_t^{[{{k}},{{j}}]}$, namely,
 $g_t^{{j}}\left(A_t^{{j}},{L}_t^{[{{j}},{{j}}]}\right) = {\psi}_t^{[{{k}},{{j}}]}\left(A_t^{{j}}\right)\left({L}_t^{[{{j}},{{j}}]}\right).$
From \eqref{eq_gen_pres}, we have
    $\Gamma_t^{[{{i}},{{j}}]} = \psi_t^{[{{i}},{{j}}]}\left(A_t^{{j}}\right)$, for all ${{j}} \in \mathcal{B}^{{i}},$
and thus,
    $\psi_t^{[{{i}},{{j}}]}\left(A_t^{{j}}\right)\left(L_t^{[{{j}},{{j}}]}\right) = g_t^{{j}}\left(A_t^{{j}},{L}_t^{[{{j}},{{j}}]}\right).$}
Hence, 
    $\psi_t^{[{{i}},{{j}}]}(A_t^{{j}})(L_t^{[{{j}},{{j}}]}) = {\psi}_t^{[{{k}},{{j}}]}(A_t^{{j}})({L}_t^{[{{j}},{{j}}]}).$

2) If ${{j}}\in \mathcal{B}^{{k}}$ and ${{j}} \not \in \mathcal{B}^{{i}}$, the control law $g_t^{{j}}$ can be constructed from the prescription law ${\psi}_t^{[{{k}},{{j}}]},$ namely,
   {$g_t^{{j}}\left(A_t^{{j}},{L}_t^{[{{j}},{{j}}]}\right) = {\psi}_t^{[{{k}},{{j}}]}\left(A_t^{{j}}\right)\left({L}_t^{[{{j}},{{j}}]}\right).$
From \eqref{eq_gen_pres}, we have
    $\Gamma_t^{[{{i}},{{j}}]} = \psi_t^{[{{i}},{{j}}]}\left(A_t^{{i}}\right)$, for all ${{j}} \not\in \mathcal{B}^{{i}}.$
Thus,
    ${\psi}_t^{[{{i}},{{j}}]}\left(A_t^{{i}}\right)\left({L}_t^{[{{j}},{{i}}]}\right) = g_t^{{j}}\left(A_t^{{i}},{L}_t^{[{{j}},{{i}}]}\right)
    = g_t^{{j}}\left(A_t^{{j}},{L}_t^{[{{j}},{{j}}]}\right).$}
Hence,
${\psi}_t^{[{{i}},{{j}}]}(A_t^{{i}})({L}_t^{[{{j}},{{i}}]})
    ={\psi}_t^{[{{k}},{{j}}]}(A_t^{{j}})({L}_t^{[{{j}},{{j}}]}).$

3) If ${{j}}\not\in\mathcal{B}^{{k}}$, the control law $g_t^{{j}}$ can be constructed from the prescription law ${\psi}_t^{[{{k}},{{j}}]},$ namely,
    {$g_t^{{j}}\left(A_t^{{j}},{L}_t^{[{{j}},{{j}}]}\right) = g_t^{{j}}\left(A_t^{{k}},{L}_t^{[{{j}},{{k}}]}\right) =
   {\psi}_t^{[{{k}},{{j}}]}\left(A_t^{{k}}\right)\left({L}_t^{[{{j}},{{k}}]}\right).$
From \eqref{eq_gen_pres}, we have
    $\Gamma_t^{[{{i}},{{j}}]} = \psi_t^{[{{i}},{{j}}]}\left(A_t^{{i}}\right)$, for all ${{j}} \not\in \mathcal{B}^{{i}}.$
Thus,
    ${\psi}_t^{[{{i}},{{j}}]}\left(A_t^{{i}}\right)\left({L}_t^{[{{j}},{{i}}]}\right) = g_t^{{j}}\left(A_t^{{i}},{L}_t^{[{{j}},{{i}}]}\right)
    = g_t^{{j}}\left(A_t^{{k}},{L}_t^{[{{j}},{{k}}]}\right).$}
Hence,
${\psi}_t^{[{{i}},{{j}}]}(A_t^{{i}})({L}_t^{[{{j}},{{i}}]}) =
    {\psi}_t^{[{{k}},{{j}}]}(A_t^{{k}})({L}_t^{[{{j}},{{k}}]}).$

To complete the proof, note that we can define a positional relationship $e^{[{{i}},{{k}}]} : \Psi^{{k}} \to \Psi^{{i}}$ with $e^{[{{i}},{{k}}]}=(e_1^{[{{i}},{{k}}]},\dots,e_T^{[{{i}},{{k}}]})$ such that \eqref{eq_e} implies the results of steps (1), (2), and (3) with ${{i}} \in \mathcal{B}^{{k}}$.
\end{proof}

\begin{lemma} \label{lem_psi_relation2}
For any given prescription strategy $\boldsymbol{\psi}^{{k}}$ of agent $k \in \mathcal{K}$, there exists a positional relationship $e^{[{{i}},{{k}}]}$, ${{i}} \not\in \mathcal{B}^{{k}}$, such that a prescription strategy $\boldsymbol{\psi}^{{i}}$ of agent ${{i}}$ generated from \eqref{eq_e} yields:
\begin{align}
    \text{\emph{1. }} {\Gamma}_t^{[{{i}},{{j}}]}(L_t^{[{{j}},{{j}}]}) &= {\Gamma}_t^{[{{k}},{{j}}]}(L_t^{[{{j}},{{j}}]}),   \quad \text{if } {{j}} \in \mathcal{B}^{{k}}, \nonumber \\
    \text{\emph{2. }} {\Gamma}_t^{[{{i}},{{j}}]}(L_t^{[{{j}},{{j}}]}) &= {\Gamma}_t^{[{{k}},{{j}}]}(L_t^{[{{j}},{{k}}]}),   \quad \text{if } {{j}} \in \mathcal{B}^{{i}}, {{j}} \not\in \mathcal{B}^{{k}}, \nonumber \\
    \text{\emph{3. }}  {\Gamma}_t^{[{{i}},{{j}}]}(L_t^{[{{j}},{{i}}]}) &= {\Gamma}_t^{[{{k}},{{j}}]}(L_t^{[{{j}},{{k}}]}),   \quad \text{if } {{j}} \not\in \mathcal{B}^{{i}}.
    \label{lem_4_condition}
\end{align}
\end{lemma}

\begin{proof}
\textcolor{black}{The proof is similar to the proof of Lemma \ref{lem_psi_relation1}, and has been omitted.}
\end{proof}

To this end, we use a relationship function $e^{[{{i}},{{k}}]}$ from every agent ${{k}} \in \mathcal{K}$ to every agent ${{i}} \in \mathcal{K}$, which satisfies Lemmas \ref{lem_psi_relation1} and \ref{lem_psi_relation2}. This implies that for any two agents ${{k}}$ and ${{i}}$, we have the relation
\begin{gather}
    U_t^{{i}} = {\Gamma}_t^{[{{i}},{{i}}]}(L_t^{[{{i}},{{i}}]}) =
    \begin{cases}
        {\Gamma}_t^{[{{k}},{{i}}]}(L_t^{[{{i}},{{k}}]}),   \quad \text{if } {{i}} \not\in \mathcal{B}^{{k}}, \\
        {\Gamma}_t^{[{{k}},{{i}}]}(L_t^{[{{i}},{{i}}]}),   \quad \text{if } {{i}} \in \mathcal{B}^{{k}}.
    \end{cases}
    \label{eq_u_with_another_psi}
\end{gather}

\subsection{The Prescription Problem}

Lemmas \ref{lem_psi_g_relation} through \ref{lem_psi_relation2} lead to \eqref{eq_u_with_another_psi}. This implies that the control action $U_t^{{i}}$ for an agent ${{i}} \in \mathcal{K}$ can be equivalently obtained through the prescription function ${\Gamma}_t^{[{{k}},{{i}}]}$ of any other agent ${{k}} \in \mathcal{K}$, if the corresponding inaccessible information were made available.
Thus, using \eqref{eq_u_with_another_psi}, we can write the cost incurred by the system from the point of view of agent $k \in \mathcal{K}$ at time $t$ as
\textcolor{black}{\begin{align}
    c_t(X_t,&U_t^1,\dots,U_t^K) \nonumber \\
    = \; &c_t\big(X_t,{\Gamma}_t^{[{{k}},1]}({L}_t^{[1,{{k}}]}),\dots,{\Gamma}_t^{[{{k}},{{k}}]}({L}_t^{[{{k}},{{k}}]}), \nonumber \\
    &{\Gamma}_t^{[{{k}},{{k+1}}]}({L}_t^{[{{k+1}},{{k+1}}]}).\dots,{\Gamma}_t^{[{{k}},K]}({L}_t^{[{{K}},K]})\big). \label{o_cost}
\end{align}}
We can then reformulate Problem 1 in terms of the prescription strategy of agent ${{k}}$.
The optimization problem for agent $k$ is to select the optimal prescription strategy $\boldsymbol{\psi}^{{*k}} \in \Psi^{{k}}$ that minimizes the performance criterion given by the total expected cost:
\textcolor{black}{\begin{multline}
\textbf{Problem 2:}~~~    \mathcal{J}^{{k}}(\boldsymbol{\psi}^{{k}}) =  \\
    \mathbb{E}^{\boldsymbol{\psi}^{{k}}} \Big[\sum_{t=0}^T{c_t\big(X_t,{\Gamma}_t^{[{{k}},1]}({L}_t^{[1,{{k}}]}),\dots,{\Gamma}_t^{[{{k}},{{k}}]}({L}_t^{[{{k}},{{k}}]})},  \\
    {\Gamma}_t^{[{{k}},{{k+1}}]}({L}_t^{[{{k+1}},{{k+1}}]}),\dots,{\Gamma}_t^{[{{k}},K]}({L}_t^{[{{K}},K]})\big)\Big]. \label{per_cri_2}
\end{multline}}


Note that this leads to a different optimization problem for every agent $k \in \mathcal{K}$. The next result shows that Problem 2 for every agent ${{k}} \in \mathcal{K}$ is equivalent to Problem 1.

\begin{lemma} \label{lem_equivalence}
For any agent ${{k}} \in \mathcal{K}$, Problem 2 is equivalent to Problem 1.
\end{lemma}

\begin{proof}
Equation \eqref{o_cost} implies that the performance criterion $\mathcal{J}^{{k}}(\boldsymbol{\psi}^{{k}})$ in \eqref{per_cri_2} is equal to the performance criterion $\mathcal{J}(\boldsymbol{g})$ in \eqref{per_cri}. Thus, given the optimal prescription strategy $\boldsymbol{\psi}^{{*k}}$, Lemmas \ref{lem_psi_g_relation} and \ref{lem_psi_g_relation_inv} imply that we can derive the optimal strategy $\boldsymbol{g}^*$ in Problem 1. 
\end{proof}

Next, we present an \textcolor{black}{equivalent state for Problem 2} for agent ${{k}}$, following the exposition presented in \cite{mahajan2008sequential}.

\begin{lemma} \label{state_suff_k}
A state sufficient for input-output mapping for agent ${{k}} \in \mathcal{K}$ is
\begin{equation}
    S^{{k}}_t := \left\{X_{t}, L_t^{[1,{{k}}]},\dots,L_t^{[{{k-1}},{{k}}]},L_t^{[{{k}},{{k}}]},\dots,L_t^{[{{K}},K]}\right\}. \label{S_k}
\end{equation}
\end{lemma}

\begin{proof}
The state $S^{{k}}_t$ satisfies the three properties stated by Witsenhausen \cite{witsenhausen1976some}:

1) There exist functions $\hat{f}^{{k}}_t$, for all $t = 0,\dots,T$ such that
  \begin{equation}
	S^{{k}}_{t+1} = \hat{f}^{{k}}_{t}(S^{{k}}_t, W_t, V_{t+1}^{1},\dots,V_{t+1}^{K}, \Theta_t^{{k}}). \label{eq_St_k1}
  \end{equation}
  
2) There exist functions $\hat{h}^{{k}}_t$, for all $t = 0,\dots,T$ such that
    \begin{equation}
	Z^{{k}}_{t+1} = \hat{h}^{{k}}_{t}(S^{{k}}_{t},\Theta_t^{{k}},V_{t+1}^{1},\dots,V_{t+1}^{K}). \label{eq_St_k2}
  \end{equation}
  
3) There exist functions $\hat{c}^{{k}}_t$, for all $t = 0,\dots,T$ such that
    \textcolor{black}{\begin{align}
	c_t(X_t,U_t^{1},\dots,U_t^{K}) &= \hat{c}^{{k}}_t(S^{{k}}_t,\Theta^{{k}}_t). \label{eq_St_k3}
  \end{align}}
The three equations above can each be verified by substitution of variables, as summarized below.

1) By definition, $S^k_{t+1} = \{X_{t+1},L_{t+1}^{[1,{{k}}]},\dots,L_{t+1}^{[{{k}},{{k}}]},\dots,$ $L_{t+1}^{[K,K]}\}$. We analyze each term individually. First, we have $X_{t+1}= f_{t}\Big(X_t,U_t^{1},\dots,U_t^{K}, W_t\Big)
    = f_{t}\Big(X_t,\Gamma_t^{[{{k}},1]}(L_t^{[1,{{k}}]}), \dots,\Gamma_t^{[{{k}},{{k}}]}(L_t^{[{{k}},{{k}}]}),\dots,\Gamma_t^{[{{k}},K]}(L_t^{[K,K]}),$ $ \; W_t\Big).$
We know that the new information added to the memory $M_t^{{k}}$ is a subset of all previously inaccessible information $\left\{L_{t}^{[1,{{k}}]},\dots,L_{t}^{[{{k}},{{k}}]},\dots,L_{t}^{[K,K]}\right\}$ and the most recent control actions and observations $\left\{Y^{1}_{t+1},U^{1}_t,\dots,Y^{K}_{t+1},U^{K}_t\right\}$.
For any $L_{t+1}^{[{{i}},{{k}}]}$, ${{i}} \in \mathcal{K}$, this implies that
    $L_{t+1}^{[{{i}},{{k}}]} \subseteq  \Big\{L_{t}^{[1,{{k}}]},\dots,L_{t}^{[{{k}},{{k}}]},\dots,L_{t}^{[K,K]}\Big\} \bigcup \Big\{Y^{1}_{t+1}, U^{1}_t,$ $\dots,Y^{K}_{t+1},U^{K}_t\Big\}.$
We can further simplify the final term as
    $\Big\{Y^{1}_{t+1},U^{1}_t,\dots,Y^{K}_{t+1},U^{K}_t\Big\}
    \subseteq \Big\{h^{1}_{t+1}(X_{t+1},V_{t+1}^{1}),$ $\dots,h^{K}_{t+1}(X_{t+1},V_{t+1}^{K}), \Gamma_t^{[{{k}},1]}(L_t^{[1,{{k}}]}),\dots,\Gamma_t^{[{{k}},{{k}}]}(L_t^{[{{k}},{{k}}]}), \dots,$ $\Gamma_t^{[{{k}},K]}(L_t^{[K,K]})\Big\}.$
This completes the proof of \eqref{eq_St_k1}.

2)  {$Z_{t+1}^{{k}} \subseteq \Big\{L_{t}^{[1,{{k}}]},\dots,L_{t}^{[{{k}},{{k}}]},$ $\dots,L_{t}^{[K,K]}\Big\} \bigcup\Big\{Y^{1}_{t+1}, U^{1}_t,\dots,Y^{K}_{t+1},U^{K}_t\Big\}$}, and 
the proof of \eqref{eq_St_k2} is completed {in a manner similar to part 1.}

3) We have already shown in part 1 that $U_t^{{i}}$ can be written as a function of $S_t^{{k}}$ and $\Theta_t^{{k}}$ for any agents two ${{k}}$ and ${{i}}$  . This completes the proof of \eqref{eq_St_k3}.
\end{proof}



\begin{lemma} \label{lem_s_relation}
Let $S_t^{{k}}$ be the state sufficient for input-output mapping of agent ${{k}} \in \mathcal{K}$. The state sufficient for input-output mapping $S_t^{{i}}$ of agent ${{i}} \in \mathcal{B}^{{k}}$ satisfies the following property
\begin{gather} \label{eq_s_relation}
    S^{{i}}_t = \Big\{S^{{k}}_t,A_t^{{k}} \setminus A_t^{{i}}\Big\}.
\end{gather}
\end{lemma}

\begin{proof}
For any ${{i}} \in \mathcal{B}^{{k}}$ and ${{j}} \not \in \mathcal{B}^{{i}}$,  we have the following properties for the inaccessible information $L_t^{[{{j}},{{i}}]}$:
\begin{align} \label{eq_lem_8_inaccess}
    \text{1. } 
    &L_t^{[{{j}},{{i}}]} = L_t^{[{{j}},{{k}}]} \cup \big(A_t^{{k}} \setminus A_t^{{i}} \big), \quad \forall {{j}} \not \in \mathcal{B}^{{k}}, \nonumber \\
    \text{2. }
    &L_t^{[{{j}},{{i}}]} = L_t^{[{{j}},{{j}}]} \cup \big(A_t^{{j}} \setminus A_t^{{i}} \big), \quad \forall {{j}} \in \mathcal{B}^{{k}}, {{j}} \not \in \mathcal{B}^{{i}},
\end{align}
Substituting \eqref{eq_lem_8_inaccess} in $S_t^{{i}}$, it holds that
\begin{align} \label{lem_8_in_between}
    S^{{i}}_t = \Big\{&X_{t}, L_t^{[1,{{k}}]}, A_t^{{k}} \setminus A_t^{{i}},\dots,L_t^{[{{k}},{{k}}]}, A_t^{{k}} \setminus A_t^{{i}}, \nonumber \\
    &\dots,L_t^{[{{i-1}},{{i-1}}]},A_t^{{i-1}} \setminus A_t^{{i}} ,L_t^{[{{i}},{{i}}]},\dots,L_t^{[{{K}},{{K}}]} \Big\} \nonumber \\
    = \Big\{&S^{{k}}_t,A_t^{{k}} \setminus A_t^{{i}},A_t^{{k+1}} \setminus A_t^{{i}},\dots,A_t^{{i-1}} \setminus A_t^{{i}} \Big\}.
\end{align}
From \eqref{ainfo_prop_1},
{$A_t^{{K}} \subseteq A_t^{{K-1}} \subseteq \dots \subseteq A_t^{{k}},$ and substituting in \eqref{lem_8_in_between} implies
$S^{{i}}_t
     = \Big\{S^{{k}}_1,A_t^{{k}} \setminus A_t^{{i}}\Big\}.$}
\end{proof}

From the point of view of any agent $k \in \mathcal{K}$, the system behaves as a Partially Observed Markov Decision Process (POMDP) with state $S_t^{{k}}$, control input $\Theta_t^{{k}}$, observation $Z_t^{{k}}$, and cost $\hat{c}^{{k}}_t(S^{{k}}_t,\Theta^{{k}}_t)$ at time $t$. The history of observations up to time $t$ is $Z^{{k}}_{0:t} = A^{{k}}_t$.
However, the prescription functions $\Gamma_t^{[{{k}},{{i}}]}$, ${{i}} \in \mathcal{B}^{{k}},$ are generated as functions of the accessible information $A_t^{{i}} \neq A_t^k$. \textcolor{black}{Note that agent $k$ utilizes different information when generating prescriptions for the other agents, implying that the problem of agent $k$ has a non-classical information structure.}
Thus, existing structural results for POMDP problems cannot be directly applied in Problem 2.

\begin{definition}
Let $S_t^{{k}}$ be the state, $A_{t}^{{k}}$ the accessible information, and $\Theta^{{k}}_{0:t-1}$  the control inputs at time $t$ for agent ${{k}}$. The \textit{information state} is a probability distribution $\Pi^{{k}}_t$ that takes values in the possible realizations $\mathscr{P}^{{k}}_t := \Delta(\mathcal{S}^{{k}}_t)$ such that,
\begin{equation}
\Pi^{{k}}_t(s^{{k}}_t) := \mathbb{P}^{\boldsymbol{\psi}^{{k}}}(S^{{k}}_t = s^{{k}}_t \big|A^{{k}}_t, \Theta^{{k}}_{0:t-1}).
\end{equation}
\end{definition}

\section{Results}

\subsection{Properties of the Information States}

In this subsection, we establish properties of the information state $\Pi_t^{{k}}$ for each agent ${{k}} \in \mathcal{K}$. The first property states that the information state $\Pi_t^{{k}}$ is independent from the prescription strategy $\boldsymbol{\psi}^{{k}}$.

\begin{lemma} \label{pi_k_1}
At time $t$, there exists a function $\tilde{f}_t^{{k}}$ independent from the prescription strategy $\boldsymbol{\psi}^{{k}}$ such that
\begin{equation}
    \Pi_{t+1}^{{k}} = \Tilde{f}_{t+1}^{{k}}(\Pi_t^{{k}},\Theta_t^{{k}},Z_{t+1}^{{k}}).
\end{equation}
\end{lemma}
\begin{proof}
See Appendix A.
\end{proof}

The second property of the information state $\Pi_t^{{k}}$ is that its evolution is Markovian.

\begin{lemma} \label{pi_k_2}
The evolution of the information state \textcolor{black}{$\Pi^{{k}}_t$} is a controlled Markov Chain with $\Theta^{{k}}_{t}$ as the control action at time $t$, \textcolor{black}{such that}
\begin{align}
\mathbb{P}(\Pi_{t+1}^{{k}}|A_t^{{k}},\Pi^{{k}}_{0:t}, \Theta^{{k}}_{0:t}) = \mathbb{P}(\Pi_{t+1}^{{k}}|\Pi^{{k}}_{t}, \Theta^{{k}}_{t}).
\end{align}
\end{lemma}
\begin{proof}
See Appendix B.
\end{proof}

The third property of the information state $\Pi_t^{{k}}$ is that the expected cost incurred by the system at time $t$ can be written as a function of $\Pi_t^{{k}}$.

\begin{lemma} \label{pi_k_3}
There exists a function $\Tilde{c}^{{k}}_t$ such that
\begin{equation}
    \mathbb{E}^{\boldsymbol{\psi}^{{k}}}\big[\hat{c}^{{k}}_t(S^{{k}}_t,\Theta^{{k}}_t)|A^{{k}}_t,\Theta^{{k}}_{0:t}\big] = \Tilde{c}_t^{{k}}(\Pi_t^{{k}},\Theta_t^{{k}}).
\end{equation}
\end{lemma}
\begin{proof}
Let $a_{t}^{{k}}$, $\theta_t^{{k}},$ and $\pi^{{k}}_t$ be the realizations of the random variables $A_{t}^{{k}}$, $\Theta_t^{{k}}$, and the conditional probability distribution $\Pi^{{k}}_t$, respectively. We expand the following expectation 
\begin{align}
    \mathbb{E}^{\boldsymbol{\psi}^{{k}}}\big[\hat{c}^{{k}}_t(& S^{{k}}_t,\Theta^{{k}}_t)|a^{{k}}_t,\theta^{{k}}_{0:t}\big] \nonumber \\
    &=\sum_{s^{{k}}_t}{\hat{c}^{{k}}_t(s^{{k}}_t,\theta^{{k}}_t)}\cdot \mathbb{P}^{\boldsymbol{\psi}^{{k}}}(S^{{k}}_t=s^{{k}}_t|a_t^{{k}},\theta_{0:t}^{{k}}) \nonumber \\
    &=\sum_{s^{{k}}_t}{\hat{c}^{{k}}_t(s^{{k}}_t,\theta^{{k}}_t)} \cdot \pi_t^{{k}}(s_t^{{k}}) =:  \Tilde{c}_t^{{k}}(\pi_t^{{k}},\theta_t^{{k}}),
\end{align}
where the term $\theta_t^{{k}}$ was dropped from the conditioning since it can be derived given $a_t^{{k}}$ and the prescription strategy $\boldsymbol{\psi}^{{k}}$.
\end{proof}

In Lemmas \ref{pi_k_1}-\ref{pi_k_3}, we established that the information state $\Pi_t^{{k}}$ evolves as a controlled Markov chain with the input $\Theta_t^{{k}}$.

\subsection{Structural Results}

We start our exposition by presenting a structural result for agent $K$. By definition, the set of agents beyond $K$ contains only agent $K$, i.e., $\mathcal{B}^{K} = \{K\}$. Using \eqref{eq_gen_pres}, this implies that for all agents ${{k}} \in \mathcal{K}$, the prescription $\Gamma^{[{{K}},{{k}}]}_t$ is a function of the accessible information $A_t^{{K}}$. This leads to the following result derived in \cite{14} using the common information approach.

\begin{lemma}\label{lem_case_K}
Consider agent $K$. There exists an optimal prescription strategy $\boldsymbol{\psi}^{*K}$ of the form
\begin{equation}
    \Gamma_t^{*[{{K}},{{k}}]} = \psi_t^{*[{{K}},{{k}}]}(\Pi_t^{K}), \label{eq_common_info}
\end{equation}
that optimizes the performance criterion \eqref{per_cri_2} in Problem 2.
\end{lemma}

For any two agents ${{k}} \in \mathcal{K}$ and ${{i}} \in \mathcal{B}^{{k}}$, we have $A_t^{{i}} \subseteq A_t^{{k}}$. Given the accessible information $A_t^{{k}}$ and the optimal prescription strategy $\boldsymbol{\psi}^{*{{k}}}$, agent ${{k}}$ can derive the optimal complete prescription $\Theta_t^{*{{i}}}$ for agent ${{i}}$ using Lemmas \ref{lem_psi_relation1} and \ref{lem_psi_relation2}. This property forms the basis of the structural result presented in this paper. \textcolor{black}{However,} we first prove the result for a two-stage system and a three-stage system, respectively, inspired by \cite{4, mahajan_online}. \textcolor{black}{The two-stage system runs for time $t=0,1$, and each agent $k \in \mathcal{K}$ selects two control actions $\{U_0^k, U_1^k\}$. Similarly, the three-stage system runs for time $t=0,1,2$, and each agent $k \in \mathcal{K}$ selects three control actions $\{U_0^k, U_1^k, U_2^k\}$. Let prescription strategy $\boldsymbol{\psi}_t^{{k}} = ({\psi}_t^{{[k,1]}}, \dots, {\psi}_t^{{[k,K]}})$}.

\begin{lemma} \label{two_stage_lemma}
\textbf{\emph{(Two-Stage Lemma)}} Consider an agent ${{k}}$ in a system with the time horizon $T=1$. If there exists an optimal prescription strategy $\boldsymbol{\psi}_1^{{*k}}$ that optimizes the performance criterion \eqref{per_cri_2} in Problem 2, it is of the form
\begin{align} \label{two_stage_eq}
\Gamma_1^{*[{{k}},{{i}}]} =
\begin{cases}
{\psi}_1^{*[{{k}},{{i}}]}(\Pi^{{k}}_1,\dots,\Pi^{K}_1), \quad \text{if } {{i}} \not\in \mathcal{B}^{{k}}, \\
{\psi}_1^{*[{{k}},{{i}}]}(\Pi^{{i}}_1,\dots,\Pi^{K}_1), \quad \text{if } {{i}} \in \mathcal{B}^{{k}}.
\end{cases}
\end{align}
\end{lemma}
\begin{proof}
See Appendix C.
\end{proof}

\begin{lemma} \label{three_stage_lemma}
\textbf{\emph{(Three-Stage Lemma)}} Consider an agent ${{k}}$ in a system with the time horizon $T=2$. Given that the optimal prescription strategy $\boldsymbol{\psi}_2^{{*k}}$ has the form
\begin{align} \label{three_stage_eq_3}
\Gamma_2^{*[{{k}},{{i}}]} =
\begin{cases}
{\psi}_2^{*[{{k}},{{i}}]}(\Pi^{{k}}_2,\dots,\Pi^{K}_2), \quad \text{if } {{i}} \not\in \mathcal{B}^{{k}}, \\
{\psi}_2^{*[{{k}},{{i}}]}(\Pi^{{i}}_2,\dots,\Pi^{K}_2), \quad \text{if } {{i}} \in \mathcal{B}^{{k}},
\end{cases}
\end{align}
if there exists an optimal prescription strategy $\boldsymbol{\psi}_1^{{*k}}$ that optimizes the performance criterion \eqref{per_cri_2} in Problem 2, it is of the form
\begin{align} \label{three_stage_eq}
\Gamma_1^{*[{{k}},{{i}}]} =
\begin{cases}
{\psi}_1^{*[{{k}},{{i}}]}(\Pi^{{k}}_1,\dots,\Pi^{K}_1), \quad \text{if } {{i}} \not\in \mathcal{B}^{{k}}, \\
{\psi}_1^{*[{{k}},{{i}}]}(\Pi^{{i}}_1,\dots,\Pi^{K}_1), \quad \text{if } {{i}} \in \mathcal{B}^{{k}}.
\end{cases}
\end{align}
\end{lemma}
\begin{proof}
See Appendix D.
\end{proof}

This leads us to the following structural result for the optimal prescription strategy $\boldsymbol{\psi}^{*{{k}}}$.

\begin{theorem} \label{struct_result}
Consider agent ${{k}} \in \mathcal{K}$. If there exists an optimal prescription strategy $\boldsymbol{\psi}^{*{{k}}}$, it is of the form
\begin{align}
\Gamma_t^{*[{{k}},{{i}}]}(\cdot) =
\begin{cases}
{\psi}_t^{*[{{k}},{{i}}]}(\Pi^{{k}}_t,\dots,\Pi_t^{K}), \quad \text{\emph{if }} {{i}} \not\in\mathcal{B}^{{k}}, \\
{\psi}_t^{*[{{k}},{{i}}]}(\Pi^{{i}}_t,\dots,\Pi_t^{K}), \quad \text{\emph{if }} {{i}} \in \mathcal{B}^{{k}},
\end{cases} \label{eq_struct_result}
\end{align}
that optimizes the performance criterion \eqref{per_cri_2} in Problem 2.
\end{theorem}

\begin{proof}
\textcolor{black}{The proof follows from Lemmas \ref{two_stage_lemma} and \ref{three_stage_lemma}.}
\end{proof}

The structural result presented in Theorem 1 can also be used to derive the optimal control strategy  $\boldsymbol{g}^* \in G$ in Problem 1 using Lemmas \ref{lem_psi_g_relation} and \ref{lem_psi_g_relation_inv}.

\begin{theorem} \label{struct_result_prob_1}
Consider agent $k \in \mathcal{K}$ in Problem 1. If there exists an optimal control strategy $\boldsymbol{g}^*$, it is of the form
\begin{align}\label{opt_cont}
U_t^{{*k}} = g_t^{{*k}}(\Pi^{{k}}_t,\dots,\Pi_t^{{{K}}},L_t^{[{{k}},{{k}}]}).
\end{align}
\end{theorem}
\begin{proof}
Let $\boldsymbol{\psi}^{*1}$ be the optimal prescription strategy for agent $1$. For any agent $k \in \mathcal{K}$, the optimal prescription strategy $\boldsymbol{\psi}^{{*k}}$ is given by \eqref{eq_e}. Every agent ${{k}} \in \mathcal{K}$ is beyond agent $1$, i.e., $ \mathcal{B}^{{1}} = \mathcal{K}$. Using \eqref{lem_3_condition}, this implies {that $\psi_t^{*[{{k}},{{k}}]}(\cdot) = \psi_t^{*[{{1}},{{k}}]}(\cdot).$}
From Lemma 1, we can select the optimal control strategy $\boldsymbol{g}^*=\Big(g_0^{{1}},\dots,g_0^{{K}},\dots,g_T^{{1}},\dots,g_T^{{K}}\Big)$ as {$U_t^{{k}} = \Gamma_t^{*[{{k}},{{k}}]}\Big(L_t^{[{{k}},{{k}}]}\Big) = \psi_t^{*[{{k}},{{k}}]}\Big(\Pi^{{k}}_t,\dots,\Pi_t^{{{K}}}\Big)\Big(L_t^{[{{k}},{{k}}]}\Big) = g_t^{{k}}\Big(\Pi^{{k}}_t,\dots,\Pi_t^{{{K}}},L_t^{[{{k}},{{k}}]}\Big).$}
\end{proof}

Note that, for all $k, i \in \mathcal{K}$ at time $t$, the structural results in Theorems 1 and 2, lead to prescription laws ${\psi}_t^{*[{{k}},{{i}}]}$ and control laws ${g}_t^k$ that have time-invariant domains, i.e., domains that stay the same size as time increases. This is in contrast to the control law $g_t^k$ of agent $k \in \mathcal{K}$ in Section-II, where the domain $\mathcal{M}_t^k$ grows in size with time.

\subsection{Illustrative Example}

\textcolor{black}{Consider the system with three agents in Section III-B.
We can define the information state for an agent ${{k}} \in \{1,2,3\}$ as follows,
\begin{align*}
    \Pi_t^{{1}} := \mathbb{P}^{\boldsymbol{\psi}^{{1}}}(&X^2_t,U^2_{t-1},X^3_{t-1:t},U^3_{t-2:t-1} \big|A^{{1}}_t, \Theta^{{1}}_{0:t-1}), \\
    \Pi_t^{{2}} := \mathbb{P}^{\boldsymbol{\psi}^{{2}}}(&X^1_t,U^1_{t-1},X^2_t,U^2_{t-1},X^3_{t-1:t},U^3_{t-2:t-1} \\
    &\big|A^{{2}}_t, \Theta^{{2}}_{0:t-1}), \\
    \Pi_t^{{3}} :=
    \mathbb{P}^{\boldsymbol{\psi}^{{3}}}(&X^1_t,U^1_{t-1},X^2_{t-1:t},U^2_{t-1:t-2},X^3_{t-1:t}, \\
    &U^3_{t-2:t-1}\big|A^{{3}}_t, \Theta^{{3}}_{0:t-1}),
\end{align*}
which can be simplified using \eqref{example_1} and the fact that the distribution of the disturbance $W_t$ is for every time step $t$ is known a priori. The simplified information states are given by
\begin{align*}
    \Pi_t^{{1}} := \mathbb{P}^{\boldsymbol{\psi}^{{1}}}(&U^2_{t-1},X^3_{t-1},U^3_{t-2:t-1} \big|A^{{1}}_t, \Theta^{{1}}_{0:t-1}), \\
    \Pi_t^{{2}} := \mathbb{P}^{\boldsymbol{\psi}^{{2}}}(&U^1_{t-1},U^2_{t-1},X^3_{t-1},U^3_{t-2:t-1} \\
    &\big|A^{{2}}_t, \Theta^{{2}}_{0:t-1}), \\
    \Pi_t^{{3}} :=
    \mathbb{P}^{\boldsymbol{\psi}^{{3}}}(&U^1_{t-1},X^2_{t-1},U^2_{t-1:t-2},X^3_{t-1},U^3_{t-2:t-1} \\
    &\big|A^{{3}}_t, \Theta^{{3}}_{0:t-1}).
\end{align*}}
Using Theorem 1, we can write the optimal prescriptions of the three agents as $\Gamma_t^{{*[1,1]}} = \psi^{{*[1,1]}}(\Pi_t^{{1}},\Pi_t^{{2}},\Pi_t^{{3}})$, $\Gamma_t^{{*[1,2]}} = \psi^{{[1,2]}}(\Pi_t^{{2}},\Pi_t^{{3}}) = \Gamma_t^{{*[2,2]}}$, and
$\Gamma_t^{{*[1,3]}} = \psi^{{[1,3]}}(\Pi_t^{{3}}) = \Gamma_t^{{*[3,3]}}$. This gives us the optimal control actions 
    $U_t^{*1} = \Gamma_t^{{*[1,1]}}$, $U_t^{*2} = \Gamma_t^{{*[2,2]}}(L_t^{{[2,2]}})$, and $U_t^{*3} = \Gamma_t^{{*[3,3]}}(L_t^{{[3,3]}})$.

\color{black}
\section{Simplification of Structural Results}
In this section, we present a simplification of the structural result for any agent $k \in \mathcal{K}$ presented in Theorem \ref{struct_result}. Consider the equivalent state $S_t^{{k}} \in \mathcal{S}_t^{{k}}$. For two agents ${{k}} \in \mathcal{K}$ and ${{i}} \in \mathcal{B}^{{k}}$, we denote the following set of random variables
\begin{gather}
    {S}_t^{[{{k}},{{i}}]} := S_t^{{i}} \setminus S_t^{{k}},
\end{gather}
that takes values in the finite collection of sets
    $\mathcal{S}_t^{[{{k}},{{i}}]} := \Big \{s_t^{{i}} \setminus s_t^{{{k}}} : s_t^{{i}} \in \mathcal{S}_t^{{{i}}}, s_t^{{k}} \in \mathcal{S}_t^{{{k}}} \Big\}.$
We also define a connection term that relates the information states $\Pi_t^{{k}}$ and $\Pi_t^{{i}}$ of agents ${{k}}$ and ${{i}}$.

\begin{definition}
Let $S_t^{{k}}$ be the equivalent state, $A_{t}^{{k}}$ the accessible information, and $\Theta^{{k}}_{0:t-1}$  the control inputs at time $t$ for an agent ${{k}} \in \mathcal{K}$. For two agents ${{k}}$ and ${{i}} \in \mathcal{B}^{{k}}$, the \textit{connection term} is a probability distribution $\Lambda^{[{{k}},{{i}}]}_t$ that takes values in the possible realizations $\mathscr{P}^{[{{k}},{{i}}]}_t := \Delta(\mathcal{S}_t^{[{{k}},{{i}}]})$, such that
\begin{multline}
\Lambda^{[{{k}},{{i}}]}_t(s^{{i}}_t \setminus s^{{k}}_t)
:= \mathbb{P}^{\boldsymbol{\psi}^{{i}}} \left({S}_t^{[{{k}},{{i}}]} = s^{{i}}_t \setminus s^{{k}}_t \; \big| \; A^{{i}}_t, \Theta^{{i}}_{0:t-1} \right). \label{eq_gamma_def}
\end{multline}
\end{definition}
Next, we show that the connection term $\Lambda^{[{{k}},{{i}}]}_t$ can be used to relate the information states $\Pi_t^{{k}}$ and $\Pi_t^{{i}}$.
\begin{lemma} \label{pi_relations}
    Let $\Pi_t^{{k}}$ be the information state of an agent ${{k}} \in \mathcal{K}$. The information state $\Pi_t^{{i}}$ of any agent ${{i}} \in \mathcal{B}^{{k}}$ satisfies the following property,
\begin{align}
    \Pi^{{i}}_t(s^{{i}}_t) = \Pi_t^{{k}}(s^{{k}}_t) \cdot \Lambda^{[{{k}},{{i}}]}_t(s^{{i}}_t \setminus s^{{k}}_t). \label{eq:pi_relation}
\end{align}
\end{lemma}
\begin{proof}
Let $a_{t}^{{k}}$, $\theta_t^{{k}}$, and $\pi^{{k}}_t$ be the realizations of the random variables $A_{t}^{{k}}$, $\Theta_t^{{k}},$ and the conditional probability distribution $\Pi^{{k}}_t$, respectively for an agent ${{k}} \in \mathcal{K}$. Let $\boldsymbol{\psi}^{{k}}$ be the prescription strategy of agent ${{k}}$. For an agent ${{i}} \in \mathcal{B}^{{k}}$, the information state at time $t$ is $\pi^{{i}}_t(s^{{i}}_t) = \mathbb{P}^{\boldsymbol{\psi}^{{i}}} \Big(S^{{i}}_t = s^{{i}}_t \; \big| \; a^{{i}}_t, \theta^{{i}}_{0:t-1} \Big).$
Since $S_t^i = S_t^k \cup S_t^{[{{k}},{{i}}]}$ and \eqref{eq_s_relation} implies that ${s}_t^{[{{k}},{{i}}]} = s_t^{{i}} \setminus s_t^{{k}} = a_t^{{k}} \setminus a_t^{{i}},$ we have
\begin{align}
    \pi^{{i}}_t(s^{{i}}_t) = \mathbb{P}^{\boldsymbol{\psi}^{{i}}} \Big(&s^{{k}}_t, s_t^{[{{k}},{{i}}]} ~\big|~ a^{{i}}_t, \theta^{{i}}_{0:t-1} \Big) \nonumber \\
    = \mathbb{P}^{\boldsymbol{\psi}^{{i}}} \Big(&S^{{k}}_t = s^{{k}}_t ~\big|~ a_t^{{k}} \setminus a_t^{{i}}, a_t^{{i}}, \theta^{{i}}_{0:t-1} \Big) \nonumber \\
    \cdot \mathbb{P}^{\boldsymbol{\psi}^{{i}}} \Big(&S_t^{[{{k}},{{i}}]} = s_t^{{i}} \setminus s_t^{{k}} ~\big|~ a^{{i}}_t, \theta^{{i}}_{0:t-1} \Big). \label{eq:first}
\end{align}
We can drop $\theta^{{i}}_{0:t-1}$ from the conditioning in the first term in \eqref{eq:first}, because it is a function of $\boldsymbol{\psi}^{{i}}$ and $a_t^{{i}}$. This implies that 
\begin{align*}
    \mathbb{P}^{\boldsymbol{\psi}^{{i}}} \Big(S^{{k}}_t = s^{{k}}_t
    ~\big|~ a_t^{{k}} \setminus a_t^{{i}}, a_t^{{i}}, \theta^{{i}}_{0:t-1} \Big) 
    = \mathbb{P}^{\boldsymbol{\psi}^{{i}}} \Big(S^{{k}}_t = s^{{k}}_t ~\big|~ a_t^{{k}} \Big). 
\end{align*}
Lemmas 3 and 4 establish the existence of a pair of positional relationship functions $e^{[{{i}},{{k}}]}(\cdot)$ and $e^{[{{k}},{{i}}]}(\cdot)$, known a priori, such that
    $\boldsymbol{\psi}^{{k}} = e^{[{{k}},{{i}}]}\big(\boldsymbol{\psi}^{{i}}\big)$ and $\boldsymbol{\psi}^{{i}} = e^{[{{i}},{{k}}]}\big(\boldsymbol{\psi}^{{k}}\big).$
This implies that
\begin{gather}
    \mathbb{P}^{\boldsymbol{\psi}^{{i}}} \Big(S_t^{{k}} = s^{{k}}_t ~\big|~ a_t^{{k}} \Big) = \mathbb{P}^{\boldsymbol{\psi}^{{k}}} \Big(S_t^{{k}} = s^{{k}}_t ~\big|~ a_t^{{k}} \Big). \label{eq:fourth}
\end{gather}
We can add the term $\theta^{{k}}_{0:t-1}$ to the conditioning of \eqref{eq:fourth} because it is simply a function of $\boldsymbol{\psi}^{{k}}$ and $a_t^{{k}}$, and thus
\begin{align}
    \mathbb{P}^{\boldsymbol{\psi}^{{k}}} \Big(S_t^{{k}} = s^{{k}}_t ~\big|~ a_t^{{k}} \Big) &=  \mathbb{P}^{\boldsymbol{\psi}^{{k}}} \Big(S_t^{{k}} = s^{{k}}_t ~\big|~ a_t^{{k}}, \theta^{{k}}_{0:t-1} \Big) \nonumber \\
    &= \pi_t^{{k}}(s_t^{{k}}). \label{eq:fifth}
\end{align}
We now substitute \eqref{eq:fifth} in \eqref{eq:first} and note that $s_t^{{i}} \setminus s_t^{{k}} = a_t^{{k}} \setminus a_t^{{i}}$
from \eqref{eq_s_relation}. This implies that
\begin{gather}
    \pi^{{i}}_t(s^{{i}}_t) = \pi_t^{{k}}(s^{{k}}_t) \cdot \mathbb{P}^{\boldsymbol{\psi}^{{i}}} \Big(S_t^{{[k,i]}} = s_t^{{i}} \setminus s_t^{{k}} ~\big|~ a^{{i}}_t, \theta^{{i}}_{0:t-1} \Big),
\end{gather}
where the second term is the realization $\lambda_t^{[{{k}},{{i}}]}(s^{{i}}_t \setminus s^{{k}}_t)$ of the connection term $\Lambda_t^{[{{k}},{{i}}]}(s^{{i}}_t \setminus s^{{k}}_t)$, and the proof is complete.
\end{proof}

Lemma \ref{pi_relations} allows us to simplify the structural result presented in Theorem \ref{struct_result}.
\begin{theorem} \label{struct_result_small}
Consider an agent ${{k}} \in \mathcal{K}$. If there exists an optimal prescription strategy $\boldsymbol{\psi}^{*{{k}}}$ that optimizes the performance criterion \eqref{per_cri_2} in Problem 2, it is of the form
\begin{align}
\Gamma_t^{*[{{k}},{{i}}]}
=
\begin{aligned}
\begin{cases}
{\psi}_t^{*[{{k}},{{i}}]}\Big(\Pi^{{k}}_t, \Lambda_t^{[{{k}},{{k+1}}]},
\dots,\Lambda_t^{[{{K-1}},{{K}}]} \Big), &\text{\emph{ if }} {{i}} \not\in\mathcal{B}^{{k}}, \\
{\psi}_t^{*[{{k}},{{i}}]} \Big(\Pi^{{i}}_t, \Lambda_t^{[{{i}},{{i+1}}]},
\dots,\Lambda_t^{[{{K-1}},{{K}}]} \Big), &\text{\emph{ if }} {{i}} \in \mathcal{B}^{{k}}.
\end{cases} \label{eq_small_struct}
\end{aligned}
\end{align}
\end{theorem}
\begin{proof}
The proof follows through a straightforward application of Lemma \ref{pi_relations} to the structural results in Theorem \ref{struct_result}. Consider agent ${{k}} \in \mathcal{K}$. For every agent ${{i}} \in \mathcal{B}^{{k}}$, we use \eqref{eq:pi_relation} to write that $\Pi_t^{{i+1}} = \Pi_t^{{{i}}} \cdot \Lambda_t^{{[{{i}},{{i+1}}]}}$. Now, starting with the information state $\Pi_t^{{K}} = \Pi_t^{{K-1}} \cdot \Lambda_t^{{[{{K-1}},{{K}}]}}$, we iteratively apply these relations in \eqref{eq_struct_result} for the agents $K, K-1, \dots, k$ to complete the proof.
\end{proof}

Note that the structural result in Theorem \ref{struct_result_small} is simpler than that in Theorem \ref{struct_result}. This is because in the argument of the optimal prescription $\Gamma_t^{*[{{k}},{{i}}]}$, the term $\Lambda_t^{[{{k}},{{i}}]}$ is smaller in size than $\Pi_t^{{i}}$ for all $i \in \mathcal{B}^{{k}}$. Thus, the argument of the optimal prescription law $\psi_t^{*[{{k}},{{i}}]}$ lies in a smaller space in \eqref{eq_small_struct} as compared to \eqref{eq_struct_result}.
\subsection{Illustrative Example}
Consider the system with three agents in Sections III-B and IV-C. At time $t$, the connection term of agent $2$ with respect to agent $1$ is given by
$
    \Lambda_t^{[{{1}},{{2}}]} = \mathbb{P}^{\boldsymbol{\psi}^{{2}}}(U^{{1}}_{t-1} | A^{{2}}_t, \Theta^{{2}}_{0: t-1}).
$
Similarly, the connection term of agent $3$ with respect to $2$ is given by
$
    \Lambda_t^{[{{2}},{{3}}]} = \mathbb{P}^{\boldsymbol{\psi}^{{3}}}(X^{{2}}_{t-1},U^{{2}}_{t-2} | A^{{3}}_t, \Theta^{{3}}_{0: t-1}).
$
Using Theorem \ref{struct_result_small}, the optimal prescriptions of the three agents are given by
\begin{align*}
    \Gamma_t^{{*[1,1]}} &= \psi^{{*[1,1]}}(\Pi_t^{{1}},\Lambda_t^{[{{1}},{{2}}]},\Lambda_t^{[{{2}},{{3}}]}), \\
    \Gamma_t^{{*[1,2]}} = \Gamma_t^{{*[2,2]}} &= \psi^{{[1,2]}}(\Pi_t^{{2}},\Lambda_t^{[{{2}},{{3}}]}) , \\
    \Gamma_t^{{*[1,3]}} = \Gamma_t^{{*[3,3]}} &= \psi^{{[1,3]}}(\Pi_t^{{3}}) .
\end{align*}
Note that the optimal prescription laws now have significantly smaller sized domains in comparison to the prescription laws in Section IV-C.
\subsection{Comparison with the Common Information Approach}
We believe that the structural result derived through the prescription approach has the advantage of compressing more information when compared with the common information approach. First, we note that the accessible information $A_t^{{K}}$ of agent $K$ is the common information among all agents in the system. Thus, the equivalent state $S_t^{{K}}$ and optimal complete prescription $\Theta_t^{{*K}}$ derived using \eqref{eq_common_info} are the same as those in the common information approach for all agents in $\mathcal{K}$.
The common information approach implies the existence of an optimal prescription strategy $\boldsymbol{\psi}^{*K}$ of the form
    $\Gamma_t^{*[{{K}},{{k}}]} = \psi_t^{*[{{K}},{{k}}]}(\Pi_t^{{K}}),$
that optimizes the performance criterion \eqref{per_cri_2} of Problem 2 for agent $K$.

Now, for two agents ${{k}} \in \mathcal{K}$ and ${{i}} \in \mathcal{B}^{{k}}$, consider the equivalent state $S_t^{{k}}$ that takes values in the finite set $\mathcal{S}_t^{{k}}$, and the random variable ${S}_t^{[{{i}},{{k}}]} = S_t^{{i}} \setminus S_t^{{k}}$ that takes values in the finite set $\mathcal{S}_t^{[{{i}},{{k}}]}$. By definition,
\begin{gather}
    S_t^{{k}} \cup S_t^{[{{k}},{{k+1}}]} \cup \dots \cup S_t^{[{{K-1}},{{K}}]} = S_t^{{K}}.
\end{gather}
Furthermore, we know from their definitions that the sets of random variables $\Big\{S_t^{{k}},S_t^{[{{k}},{{k+1}}]}, \dots, S_t^{[{{K-1}},{{K}}]} \Big\}$ are pairwise disjoint, i.e., the intersection of any two sets is $\emptyset$. As an example, consider $S_t^{[{{K-1}},{{K}}]}$ and $S_t^{{k}}$. We know from \eqref{eq_s_relation} that $S_t^{{k}} \subseteq S_t^{{K-1}},$
and thus, $S_t^{[{{K-1}},{{K}}]} \cap S_t^{{k}} \nonumber = \Big(S_t^{{K}} \setminus S_t^{{K-1}} \Big) \cap S_t^{{k}} \subseteq  \Big(S_t^{{K}} \setminus S_t^{{k}} \Big) \cap S_t^{{k}} = \emptyset.$
Thus, the cardinalities of these sets of random variables are related as follows,
\begin{gather}
    \big|S_t^{{k}}\big| + \big|S_t^{[{{k}},{{k+1}}]}\big| + \dots + \big|S_t^{[{{K-1}},{{K}}]}\big| = \big|S_t^{{K}}\big|,
\end{gather}
The argument of the optimal prescription law ${\psi}_t^{*[{{k}},{{i}}]}(\cdot)$ in the common information approach is the distribution $\mathbb{P}^{\boldsymbol{\psi}^k}(S_t^{{k}},S_t^{[{{k}},{{k+1}}]}, \dots, S_t^{[{{K-1}},{{K}}]}|A_t^K,\Theta_{0:t-1}^{K})$ that takes values in the space $\Delta(\mathcal{S}_t^{{K}})$. In contrast, the argument of the optimal prescription law ${\psi}_t^{*[{{k}},{{i}}]}(\cdot)$ in \eqref{eq_small_struct} is
\begin{gather*}
    \Big(\mathbb{P}^{\boldsymbol{\psi}^k}(S^{{k}}_t | A_t^k, \Theta_t^k), \mathbb{P}^{\boldsymbol{\psi}^{k+1}}(S^{[{{k}},{{k+1}}]}_t | A_t^{k+1}, \Theta_t^{k+1}), \\
    \dots,\mathbb{P}^{\boldsymbol{\psi}^{K}}(S^{[{{K-1}},{{K}}]}_t | A_t^{K}, \Theta_t^{K})\Big),
\end{gather*}
which takes values in the space $\Delta(\mathcal{S}_t^{{k}}) \times \Delta(\mathcal{S}_t^{[{{k}},{{k+1}}]}) \times \dots \times \Delta(\mathcal{S}_t^{[{{K-1}},{{K}}]})$. The spaces in which the arguments of both prescriptions take their values are of the same dimension.

\color{black}

Meanwhile, consider the optimal control action $U_t^{{*k}}$ for agent ${{k}}$ generated through the common information approach as
    $U_t^{{*k}} = \Gamma_t^{*[{{K}},{{k}}]}(L_t^{[{{k}},{{K}}]}).$
The same optimal control action $U_t^{{*k}}$, when generated through the prescription approach, is given by,
$U_t^{{*k}} = \Gamma_t^{*[{{k}},{{k}}]}(L_t^{[{{k}},{{k}}]}).$
From the definition of inaccessible information, we note that
$
    M_t^{{k}} = \{A_t^{{k}},L_t^{[{{k}},{{k}}]}\} = \{A_t^{{K}},L_t^{[{{k}},{{K}}]}\}$,
and from \eqref{ainfo_prop_1}, we know that $
    A_t^{{K}} \subseteq A_t^{{k}}$. Thus,
\begin{gather}
    L_t^{[{{k}},{{k}}]} \subseteq L_t^{[{{k}},{{K}}]}.
\end{gather}
Thus, our prescription functions have a smaller sized domain than those generated through the common information approach. \textcolor{black}{This implies that the prescription approach leads to a significantly smaller space of optimal prescription strategies as compared to the common information approach.}

\textcolor{black}{Specifically, agent $K$ has the largest domain for her prescriptions, and subsequently, requires the most computational effort to compute the optimal prescription strategy $\boldsymbol{\psi}^{*K}$. The performance of agent $K$ is the same as the common information approach. In contrast, agent $1$ has the smallest domain for her prescriptions and thus, requires the least computational effort to compute the optimal prescription strategy $\boldsymbol{\psi}^{*1}$. For the rest of the agents, the computational effort for optimal strategies increases with the index of the agent.}

\section{Discussion and Concluding Remarks}

In this paper, we analyzed a decentralized stochastic control problem with a word-of-mouth information structure and established structural results for optimal control laws with time-invariant domain sizes. We presented these structural results by formulating an alternative problem through the prescription approach. We showed that the result of the common information approach is a special case of the prescription approach, and that it is possible to compress larger amounts of data in asymmetric information sharing structures like the word-of-mouth information structure. \textcolor{black}{Furthermore, we showed that our result leads to a smaller search space for optimal strategies in comparison to the state of the art.}

\color{black}

\subsection{Aspects of Computation and Implementation}

We note in Section V-B that agent $1$ requires the least number of computations to derive her optimal prescription strategy $\boldsymbol{\psi}^{{*1}}$. This is because, for every agent $k \in \mathcal{K}$, agent $1$ uses the entire accessible information $A_t^k$ to generate the prescription $\Gamma_t^{[{{1}},{{k}}]}$, and thus compresses the maximum information from the memory of agent $k$. We have also noted in Remark 5 that the optimal control action corresponding to the optimal prescription of any agent remains the same, irrespective of the index assigned to that agent. However, despite achieving the same optimal control actions and optimal cost for all indexing patterns, the computational performance of the prescription approach changes with the indexing pattern. 

As an illustration, consider the static system in Section I-C. We re-index the agents in this system such that the information available to agents $1$, $2$, and $3$ is $M^1 = \{Y^3\}$, $M^2 = \{Y^2, Y^3\}$, and $M^3 = \{Y^1, Y^2, Y^3\}$, respectively. Witsenhausen's result on ordering agents \cite{witsenhausen1971information} implies that the optimal cost achieved by this re-indexed system is equal to the optimal cost achieved by the system in Section I-C. However, in the new system, the accessible information of all three agents is given by $A^1 = A^2 = A^3 = \{Y^3\}$. From the analysis in Section I-C, the number of computations required to derive the optimal prescriptions is $256$ for every agent $k \in \{1,2,3\}$. This is the same as the number of computations in the common information approach, and equivalent to the worst case performance of the prescription approach.

We believe that this challenge can be overcome using the following indexing pattern:
(1) The agent $1$ is selected such that,
$
    |M_t^1| \geq |M_t^k|$, for all $k \in \mathcal{K},
$
at any time $t$.
(2) Subsequently, the agent $k \in \mathcal{K} \setminus \{1\}$ is indexed such that,
\begin{align*}
        \Big|\Big(\bigcap_{{{j}}=1}^{k-1}M_t^{{{j}}}\Big) \cap M_t^{{{k}}}\Big| \geq \Big|\Big(\bigcap_{{{j}}=1}^{k-1}M_t^{{{j}}}\Big) \cap M_t^{{{i}}}\Big|,
        \forall {{i}} \in \mathcal{B}^{{{k}}},
\end{align*}
at any time $t$. For example, the memory of agent $2$ satisfies the property
$
    |M_t^{{{1}}} \cap M_t^{{{2}}}| \geq |M_t^{{{1}}} \cap M_t^{{{i}}}|$ for all ${{i}} \in \{2,\dots,K\}$. 

In this indexing pattern, (1) indicates that agent $1$ has the most information in the system and (2) ensures that any agent $k \in \mathcal{K} \setminus \{1\}$ has the largest accessible information $A_t^k$ among the remaining agents. However, we believe that there may be more efficient indexing patterns for agents that maximize the accessible information of all agents, and subsequently, maximize the computational efficiency.

\color{black}

\subsection{Ongoing and Future Research Directions}

We observed that the prescription approach does not depend on the information structure of the system beyond the definition of the memory of agents. This indicates that it should be possible to translate these results to different information structures without being constrained by the assumptions made for word-of-mouth communication. \textcolor{black}{Ongoing work includes exploring the application of the prescription approach to systems with no common information across all agents.}

\textcolor{black}{Finally, as we indicated in Section VI-A, the computational performance of the prescription approach depends on the indexing pattern of the agents. An interesting research direction for future research is the derivation of optimal indexing patterns in terms of computational performance. We also plan to explore the possibility of dynamically re-indexing agents in systems with dynamic communication networks.}

\section*{Appendix A\\Proof of Lemma \ref{pi_k_1}}

Let $a_{t+1}^{{k}}$, $\theta_t^{{k}},$ and $\pi^{{k}}_t$ be the realizations of the random variables $A_{t+1}^{{k}}$, $\Theta_t^{{k}},$ and the conditional probability distribution $\Pi^{{k}}_t$, respectively. Let $\boldsymbol{\psi}^{{k}}$ be the prescription strategy. Then the realization of $\Pi^{{k}}_{t+1}$ is
\begin{equation}
    \pi^{{k}}_{t+1}(s_{t+1}) = \mathbb{P}^{\boldsymbol{\psi}^{{k}}}(S^{{k}}_{t+1} = s^{{k}}_{t+1}|a^{{k}}_{t+1},\theta_{0:t}^{{k}}). \label{ap_a_1}
\end{equation}
Since $a_{t+1}^k = a_t^k \cup z_t^k$, the state $S^{{k}}_t$ in \eqref{ap_a_1} can be written as
\begin{gather}
    \pi^{{k}}_{t+1}(s_{t+1})= \sum_{s^{{k}}_t,v_{t+1}^{{1:K}},w_t} \mathbb{I}_{s^{{k}}_{t+1}}(\Hat{f}^{{k}}_{t+1}(s^{{k}}_t, w_t, v_{t+1}^{{1:K}}, \theta_t^{{k}})) \nonumber \\
    \cdot\mathbb{P}(V_{t+1}^{{1:K}}=v_{t+1}^{{1:K}}) \cdot\mathbb{P}(W_t=w_t)
    \cdot\mathbb{P}^{\boldsymbol{\psi}^{{k}}}(S^{{k}}_t=s^{{k}}_t|a^{{k}}_{t+1},\theta_{0:t}^{{k}}). \label{ap_a_2}
\end{gather}
The last term in \eqref{ap_a_2} can be written as
\begin{align} 
    &\mathbb{P}^{\boldsymbol{\psi}^{{k}}}(S^{{k}}_t=s^{{k}}_t|a^{{k}}_{t+1},\theta_{0:t}^{{k}}) \nonumber \\
        & = \frac{\left[\splitfrac{\mathbb{P}^{\boldsymbol{\psi}^{{k}}}(Z^{{k}}_{t+1}=z^{{k}}_{t+1}|s^{{k}}_t,a^{{k}}_{t},\theta_{0:t}^{{k}})}
        {\cdot\mathbb{P}^{\boldsymbol{\psi}^{{k}}}(s^{{k}}_t, a^{{k}}_{t},\theta_{0:t}^{{k}})}\right]}
        {\sum_{\tilde{s}^k_t}\left[{\splitfrac{\mathbb{P}^{\boldsymbol{\psi}^{{k}}}(Z^{{k}}_{t+1}=z^{{k}}_{t+1}|\tilde{s}^k_t,a^{{k}}_{t},\theta_{0:t}^{{k}})}
        {\cdot\mathbb{P}^{\boldsymbol{\psi}^{{k}}}(\tilde{s}^k_t, a^{{k}}_{t}, \theta_{0:t}^{{k}})}}\right]}. \label{ap_a_3}
\end{align}
However,
\begin{align}
    &\mathbb{P}^{\boldsymbol{\psi}^{{k}}}(s^{{k}}_t, a^{{k}}_{t}, \theta_{0:t}^{{k}}) \nonumber \\
    = & \mathbb{P}^{\boldsymbol{\psi}^{{k}}}(S^{{k}}_t = s^{{k}}_t|a^{{k}}_{t}, \theta_{0:t-1}^{{k}}) \cdot \mathbb{P}^{\boldsymbol{\psi}^{{k}}}(a^{{k}}_{t},\theta_{0:t}^{{k}}). \label{ap_a_4}
\end{align}
We can drop $\theta^{{k}}_t$ from the conditioning in \eqref{ap_a_4} since $\theta^{{k}}_t$ can be derived given $a_t^{{k}}$ and $\boldsymbol{\psi}^{{k}}$. Substituting \eqref{ap_a_4} in \eqref{ap_a_3}, we have
\begin{align}
    \mathbb{P}^{\boldsymbol{\psi}^{{k}}}&(S^{{k}}_t=s^{{k}}_t|a^{{k}}_{t+1},\theta_{0:t}^{{k}}) \nonumber \\
    =&\frac{\left[\splitfrac{\mathbb{P}^{\boldsymbol{\psi}^{{k}}}(Z^{{k}}_{t+1}=z^{{k}}_{t+1}|s^{{k}}_t,a^{{k}}_{t},\theta_{0:t}^{{k}})}
    {\cdot\mathbb{P}^{\boldsymbol{\psi}^{{k}}}(S^{{k}}_t = s^{{k}}_t|a^{{k}}_{t}, \theta_{0:t-1}^{{k}})}\right]}
    {\sum_{\tilde{s}^k_t}{\left[\splitfrac{\mathbb{P}^{\boldsymbol{\psi}^{{k}}}(Z^{{k}}_{t+1}=z^{{k}}_{t+1}|\tilde{s}^k_t,a^{{k}}_{t},\theta_{0:t}^{{k}})}
    {\cdot\mathbb{P}^{\boldsymbol{\psi}^{{k}}}(S^{{k}}_t = \tilde{s}^k_t|a^{{k}}_{t}, \theta_{0:t-1}^{{k}})}\right]}} \nonumber \\
    =&\frac{\mathbb{P}^{\boldsymbol{\psi}^{{k}}}(Z^{{k}}_{t+1}=z^{{k}}_{t+1}|s^{{k}}_t,a^{{k}}_{t},\theta_{0:t}^{{k}})\cdot\pi^{{k}}_t(s_t^{{k}})}
    {\sum_{\tilde{s}^k_t}{\mathbb{P}^{\boldsymbol{\psi}^{{k}}}(Z^{{k}}_{t+1}=z^{{k}}_{t+1}|\tilde{s}^k_t,a^{{k}}_{t},\theta_{0:t}^{{k}})\cdot\pi^{{k}}_t(\tilde{s}^k_t)}}.&
    \label{ap_a_5}
\end{align}
However,
\begin{align*}
    \mathbb{P}^{\boldsymbol{\psi}^{{k}}}(Z^{{k}}_{t+1}=z^{{k}}_{t+1}|&s^{{k}}_t,a^{{k}}_{t},\theta_{0:t}^{{k}}) = \mathbb{I}_{\Hat{h}^{{k}}_t(s^{{k}}_{t},\theta^{{k}}_t,v_{t+1}^{{1:K}})}\big(z^{{k}}_{t+1}\big),
\end{align*}
and substituting in \eqref{ap_a_5}, we have
\begin{align}
    \mathbb{P}^{\boldsymbol{\psi}^{{k}}}(&S^{{k}}_t=s^{{k}}_t|a^{{k}}_{t+1},\theta_{0:t}^{{k}}) \nonumber \\
    &=\frac{\mathbb{I}_{\Hat{h}^{{k}}_t(s^{{k}}_{t},\theta^{{k}}_t, v_{t+1}^{{1:K}})}\big(z^{{k}}_{t+1}\big) \cdot\pi^{{k}}_t(s_t^{{k}})}
    {\sum_{\tilde{s}^k_t}{\mathbb{I}_{\Hat{h}^{{k}}_t(\tilde{s}^k_t,\theta^{{k}}_t,v_{t+1}^{{1:K}})}\big(z^{{k}}_{t+1}\big)\cdot\pi^{{k}}_t(\tilde{s}^k_t)}}. \label{ap_a_6}
\end{align}
From \eqref{ap_a_2} and \eqref{ap_a_6}, the proof is complete.

\section*{Appendix B\\Proof of Lemma \ref{pi_k_2}}

Let $a_{t}^{{k}}$, $\theta_t^{{k}},$ and $\pi^{{k}}_t$ be the realizations of the random variables $A_{t}^{{k}}$, $\Theta_t^{{k}}$, and the conditional probability distribution $\Pi^{{k}}_t$, respectively. Let $\boldsymbol{\psi}^{{k}}$ be the prescription strategy. Then, for some Borel subset $P \subseteq \mathscr{P}_t^{{k}}$,
\begin{align}
    \mathbb{P}(\Pi_{t+1}^{{k}} \in P| a_t^k, \theta^{{k}}_{0:t},\pi^{{k}}_{0:t})
    =\sum_{z^{{k}}_{t+1}}&\mathbb{I}_P\left[\tilde{f}_{t+1}^{{k}}(\Pi_t^{{k}},\theta_{t}^{{k}},z_{t+1}^{{k}})\right] \nonumber \\
    &\cdot\mathbb{P}(Z^{{k}}_{t+1}=z^{{k}}_{t+1}|a_t^{{k}},\theta_t^{{k}},\pi_{0:t}^{{k}})\label{proof_2_1}.
\end{align}
Then, the second term in \eqref{proof_2_1} can be expanded
\begin{align}
    \mathbb{P}&(Z^{{k}}_{t+1}=z^{{k}}_{t+1}|a_t^{{k}},\theta_{0:t}^{{k}},\pi_{0:t}^{{k}}) \nonumber \\
    &=\sum_{s_t^{{k}}, v_{t+1}^{1:K}}\mathbb{I}_{\hat{h}_t^{{k}}(s_t^{{k}},\theta_{0:t}^{{k}},v_{t+1}^{{1:K}})}(z_{t+1}^{{k}})\cdot\mathbb{P}(V_{t+1}^{{1:K}}=v_{t+1}^{{1:K}}) \nonumber \\
    & \quad \quad \quad \quad \cdot\mathbb{P}(S_t^{{k}}=s_t^{{k}}|a_t^{{k}},\theta_{0:t}^{{k}},\pi_{0:t}^{{k}}) \nonumber \\
    &=\sum_{s_t^{{k}}, v_{t+1}^{1:K}}\mathbb{I}_{\hat{h}_t^{{k}}(s_t^{{k}},\theta_{0:t}^{{k}},v_{t+1}^{{1:K}})}(z_{t+1}^{{k}}) \cdot\mathbb{P}(V_{t+1}^{{1:K}}=v_{t+1}^{{1:K}}) \nonumber \\
    &\quad \quad \quad \quad \cdot \pi_t^{{k}}(s_t^{{k}})
    \label{proof_2_1a}.
\end{align}
The last equality in \eqref{proof_2_1a} holds since given the realizations of the accessible information $a_t^{{k}}$ and prescription strategy $\boldsymbol{\psi}^{{k}}$, the realization of the complete prescription $\theta_t^{{k}}$ is determined. Substituting \eqref{proof_2_1a} into \eqref{proof_2_1}, the proof is complete.

\color{black}

\section*{Appendix C\\Proof of Lemma \ref{two_stage_lemma}}

We use mathematical induction to develop the proof in four steps:

\textit{Step 1.} We show that the result holds for agent $K$.

\textit{Step 2.} We assume that the result holds for agent $k+1 \in \mathcal{K}$.

\textit{Step 3.} Using the assumption in Step 2, we prove that the result holds for agent ${{k}} \in \mathcal{K}$.

\textit{Step 4.} Starting with agent $K$, we can prove the result for agents $K$, $K-1$, $\dots$, ${{1}}$ using mathematical induction.

Let prescription strategy $\boldsymbol{\psi}_t^{{k}} = ({\psi}_t^{[{{k}},{{1}}]}, \dots, {\psi}_t^{[{{k}},{{K}}]})$. Now, we present the fours steps in detail.

\textit{Step 1.} Lemma \ref{lem_case_K} establishes the result for agent $K$.

\textit{Step 2.} Let $\boldsymbol{\psi}^{{*k+1}}_1$ be an optimal prescription strategy for agent $k+1$, such that
\begin{align*} 
\Gamma_1^{*[{{k+1}},{{i}}]} =
\begin{cases}
{\psi}_1^{*[{{k+1}},{{i}}]}(\Pi^{{k+1}}_1,\dots,\Pi^{K}_1), \text{ if } {{i}} \not\in \mathcal{B}^{{k+1}}, \\
{\psi}_1^{*[{{k+1}},{{i}}]}(\Pi^{{i}}_1,\dots,\Pi^{K}_1), \quad \text{ if } {{i}} \in \mathcal{B}^{{k+1}}.
\end{cases}
\end{align*}

\textit{Step 3.1.} For agent ${{k}} \in \mathcal{K}$, from Lemma \ref{lem_psi_relation1} we have the relation
    $\psi^{[{{k}},{{j}}]}_1 (\cdot) = \psi^{[{{i}},{{j}}]}_1 (\cdot)$, for all $\forall {{i}} \in \mathcal{B}^{{k}}$, and ${{j}} \in \mathcal{B}^{{i}}.$
Here, we substitute ${{i}} = {{k+1}}$ to obtain the relation
    $\psi_1^{*[{{k}},{{j}}]}(\cdot) = \psi_1^{*[{{k+1}},{{j}}]}(\cdot)$, for all ${{j}} \in \mathcal{B}^{{k+1}}$, which implies that
\begin{gather}
    \Gamma_1^{*[{{k}},{{j}}]} = \psi_1^{*[{{k}},{{j}}]}(\Pi^{{j}}_1,\dots,\Pi^{K}_1), \quad \forall {{j}} \in \mathcal{B}^{{k+1}}. \label{ap_c_lem_1}
\end{gather}
Note that agent ${{k}}$ can derive the information state $\Pi_t^{{j}}$ for every ${{j}} \in \mathcal{B}^{{k+1}}$ since $A_t^{{j}} \subseteq A_t^{{k}}$, and $\psi_t^{[{{j}},{{i}}]} = e_t^{[{{j}},{{k}}]}\big(\psi_t^{[{{k}},{{i}}]}\big)$ for all $i \in \mathcal{K}$.

\textit{Step 3.2.} For agent ${{k}} \in \mathcal{K}$, we fix the prescription strategy $\boldsymbol{\psi}_0^{{k}}$. Then, in the total expected cost incurred by the system, given by
\begin{gather*}
    \mathbb{E}^{\boldsymbol{\psi}^{{k}}}\big[c_0^{{k}}(S_0^{{k}},\Theta_0^{{k}}) + c_1^{{k}}(S_1^{{k}},\Theta_1^{{k}})\big], 
\end{gather*}
the choice of the prescription strategy $\boldsymbol{\psi}_1^{{k}}$ affects only the second term. Let $a_t^k$ and $\theta_{t}^k$ be the realization of the accessible information $A_t^k$ and complete prescription $\Theta_{t}^k$, respectively, of agent $k$ at time $t$. Then, the cost-to-go for agent $k$ at time $t=1$ is defined as
\begin{gather}
    \mathcal{J}_1^k : =\mathbb{E}^{\boldsymbol{\psi}^{{k}}}\big[ c_1^{{k}}(S_1^{{k}},\Theta_1^{{k}}) | a_1^k, \theta_{0}^k\big].
\end{gather}
Let $\pi_1^k$ be the realization of the the information state $\Pi_1^k$. Note that, the information state $\pi_1^k$ is simply a function of the accessible information $a_1^k$ and complete prescription $\theta_{0}^k.$ Now, we can use Lemma \ref{pi_k_3} to write that
\begin{gather}
    \mathcal{J}_t^k = \Tilde{c}_1^{{k}}(\pi_1^k,\theta_1^k).
\end{gather}
Using \eqref{ap_c_lem_1} in Step 3.1, we have $\Tilde{c}_1^{{k}}(\pi_1^{{k}},\theta_1^{{k}}) = \Tilde{c}_1^{{k}}\big(\pi_1^{{k}},\gamma_1^{{[k,1]}},$ $\dots,\gamma_1^{{[k,k]}},\psi_1^{{*[k,k+1]}}(\pi^{{{k+1}}:{{K}}}_1), \nonumber 
    \dots,\psi_1^{{*[k,K]}}(\pi_1^{K})\big),$
where $\pi_t^{{k:K}} = (\pi_t^{{k}},\dots,\pi_t^{{K}})$. Then, the optimal prescriptions $\gamma_1^{*[{{k}},{{i}}]}$ of agent ${{k}}$ for all agents ${{i}} \not \in \mathcal{B}^{{k+1}}$ are defined as
\begin{multline}
    ({\gamma}_1^{*[{{k}},{{1}}]},\dots,{\gamma}_1^{*[{{k}},{{k}}]}) 
    %
    := \arginf_{{\gamma}_1^{[{{k}},{{1}}]},\dots,{\gamma}_1^{[{{k}},{{k}}]}}
    \Big\{{C}_1^{{k}}\Big({\pi}_1^{{k}},{\gamma}_1^{[{{k}},{{1}}]},\dots, \nonumber \\
    {\gamma}_1^{[{{k}},{{k}}]}, {\psi}_1^{*[{{k}},{{k+1}}]}(\pi_1^{{k+1:K}}), \dots,{\psi}_1^{*[{{k}},{{K}}]}(\pi_1^{{K}})\Big)\Big\}.
\end{multline}
Thus, there exists a function ${\psi}_1^{*[{{k}},{{i}}]}: \mathscr{P}_t^{{k}} \times \dots \times \mathscr{P}_t^{{K}} \to \mathscr{G}_t^{[{{k}},{{i}}]}$ for every ${{i}} \not \in \mathcal{B}^{{k+1}}$ that gives us the optimal prescription
    ${\gamma}_1^{*[{{k}},{{i}}]} = {\psi}_1^{*[{{k}},{{i}}]}(\pi_1^{{k}},\dots,\pi_1^{{K}}).$

\textit{Step 4.} By mathematical induction starting with agent $K$, the optimal prescription strategy $\boldsymbol{\psi}_1^{*{{k}}}$ for agent ${{k}} \in \mathcal{K}$ satisfies the structural result
\begin{align}
\Gamma_1^{*[{{k}},{{i}}]} =
\begin{cases}
{\psi}_1^{*[{{k}},{{i}}]}(\Pi^{{k}}_1,\dots,\Pi^{K}_1), \quad \text{if } {{i}} \not\in \mathcal{B}^{{k}}, \\
{\psi}_1^{*[{{k}},{{i}}]}(\Pi^{{i}}_1,\dots,\Pi^{K}_1), \quad \text{if } {{i}} \in \mathcal{B}^{{k}}.
\end{cases}
\end{align}

\section*{Appendix D\\Proof of Lemma \ref{three_stage_lemma}}

We use mathematical induction to develop the proof in five steps:

\textit{Step 1.} We show that the result holds for agent $K$.

\textit{Step 2.} We assume that the result holds for \textit{all} agents $k+1, k+2,\dots,K$.

\textit{Step 3.} With the assumption in Step 2, we prove that the result holds for the prescription $\Gamma_1^{{[{{k}},{{j}}]}}$ of agent ${{k}} \in \mathcal{K}$ for an agent ${{j}} \in \mathcal{B}^{{k+1}}$.

\textit{Step 4.} With the assumption in Step 2, we prove that the result holds for the prescription $\Gamma_1^{{[{{k}},{{j}}]}}$ of agent ${{k}} \in \mathcal{K}$ for an agent ${{j}} \not \in \mathcal{B}^{{k+1}}$.

\textit{Step 5.} Starting with agent $K$ we can prove the result for agents $K$, $K-1$, $\dots$, ${{1}}$ using mathematical induction.

Let prescription strategy $\boldsymbol{\psi}_t^{{k}} = ({\psi}_t^{[{{k}},{{1}}]}, \dots, {\psi}_t^{[{{k}},{{K}}]})$. Next, we present the five steps in detail.

\textit{Step 1.} Lemma \ref{lem_case_K} establishes the result for agent $K$.

\textit{Step 2.} We let $\boldsymbol{\psi}^{{*j}}_1$ be an optimal prescription strategy of every agent ${{j}} \in \mathcal{B}^{{k+1}}$, such that
\begin{align} \label{eq_ap_d_induction}
\Gamma_1^{*[{{j}},{{i}}]} :=
\begin{cases}
{\psi}_1^{*[{{j}},{{i}}]}(\Pi^{{j}}_1,\dots,\Pi^{K}_1), \quad \text{if } {{i}} \not\in \mathcal{B}^{{j}}, \\
{\psi}_1^{*[{{j}},{{i}}]}(\Pi^{{i}}_1,\dots,\Pi^{K}_1), \quad \text{if } {{i}} \in \mathcal{B}^{{j}}.
\end{cases}
\end{align}

\textit{Step 3.} For agent ${{k}} \in \mathcal{K}$, from Lemma \ref{lem_psi_relation1} we have the relation
    $\psi^{[{{k}},{{j}}]}_1 (\cdot) = \psi^{[{{i}},{{j}}]}_1 (\cdot)$, for all ${{i}} \in \mathcal{B}^{{k}}$, and ${{j}} \in \mathcal{B}^{{i}}.$
Here, we substitute ${{i}} = {{k+1}}$ to obtain $\psi_1^{*[{{k}},{{j}}]}(\cdot) = \psi_1^{*[{{k+1}},{{j}}]}(\cdot)$ for all ${{j}} \in \mathcal{B}^{{k+1}}$. This implies that
\begin{gather}
    \Gamma_1^{*[{{k}},{{j}}]} = \psi_1^{*[{{k}},{{j}}]}(\Pi^{{j}}_1,\dots,\Pi^{K}_1), \quad \forall {{j}} \in \mathcal{B}^{{k+1}}. \label{ap_c_lem_1_1}
\end{gather}
Note that agent ${{k}}$ can derive the information state $\Pi_t^{{j}}$ for every ${{j}} \in \mathcal{B}^{{k+1}}$ since $A_t^{{j}} \subseteq A_t^{{k}}$, and $\psi_t^{[{{j}},{{i}}]} = e_t^{[{{j}},{{k}}]}\big(\psi_t^{[{{k}},{{i}}]}\big)$ for all ${{i}} \in \mathcal{K}$.

\textit{Step 4.1.} For agent ${{k}} \in \mathcal{K}$, we fix the prescription strategy $\boldsymbol{\psi}_0^{{k}}$ and select the prescription strategy $\boldsymbol{\psi}_2^{{*k}}$ using \eqref{three_stage_eq_3}. In the total expected cost incurred by the system, given by
\begin{align*}
    \mathbb{E}^{\boldsymbol{\psi}^{{k}}}\Big[c_0^{{k}}(S_0^{{k}},\Theta_0^{{k}})  +c_1^{{k}}(S_1^{{k}},\Theta_1^{{k}})  +c_2^{{k}}(S_2^{{k}},\Theta_2^{{k}})\Big], 
\end{align*}
the choice of the strategy $\boldsymbol{\psi}_1^{{k}}$ affects only the second and third terms. Let $a_t^k$ and $\theta_{t}^k$ be the realization of the accessible information $A_t^k$ and complete prescription $\Theta_{t}^k$, respectively, of agent $k$ at time $t$. The cost-to-go for agent $k$ at time $t=1$ is given by
\begin{gather}
    \mathcal{J}^k_1 := \mathbb{E}^{\boldsymbol{\psi}^{{k}}}\Big[c_1^{{k}}(S_1^{{k}},\Theta_1^{{k}})  +c_2^{{k}}(S_2^{{k}},\Theta_2^{{k}}) \; | \; a_1^k, \theta_{0}^k\Big] \label{ap_d_to_go}
\end{gather}

\textit{Step 4.2.} Let $s_t^{{k:K}} = (s_t^{{k}},\dots,s_t^{{K}})$, $\theta_t^{{k:K}} = (\theta_t^{{k}},\dots,\theta_t^{{K}})$ and $\pi_t^{{k:K}} = (\pi_t^{{k}},\dots,\pi_t^{{K}})$ be the realizations of the random variables $S_t^{{k:K}} = (S_t^{{k}},\dots,S_t^{{K}})$, $\Theta_t^{{k:K}} = (\Theta_t^{{k}},\dots,\Theta_t^{{K}})$, and the conditional probability distribution $\Pi_t^{{k:K}} = (\Pi_t^{{k}},\dots,\Pi_t^{{K}})$.
Consider the second term in \eqref{ap_d_to_go}, $\mathbb{E}^{\boldsymbol{\psi}^{{k}}}\Big[c_2^{{k}}(S_2^{{k}},\Theta_2^{{k}}) | a_1^k, \theta_{0}^k\Big].$ 
We can substitute the structural result for $\boldsymbol{\psi}_2^{{*k}}$ in \eqref{three_stage_eq_3} to obtain the relation
\begin{multline} \label{ap_c_lem_1_3}
   c_2^{{k}}(S_2^{{k}},\Theta_2^{{k}})
    = c_2^{{k}}\Big(S_2^{{k}},\psi_2^{*[{{k}},{{1}}]}(\Pi_2^{{k:K}}),\dots,\psi_2^{*[{{k}},{{k}}]}(\Pi_2^{{k:K}}), \\
    \dots,\psi_2^{*[{{k}},{{K}}]}(\Pi_2^{{K}})\Big).
\end{multline}
For agent ${{k}} \in \mathcal{K}$, we invoke Lemmas \ref{state_suff_k} and \ref{pi_k_1} to state the relations
\begin{align}
    \text{1. }
    &\Pi_2^{{k}} = \Tilde{f}_2^{{k}}(\Pi_1^{{k}},\Theta_1^{{k}},Z_2^{{k}}), \nonumber \\
    \text{2. }
    &Z_2^{{k}} = \hat{h}_1^{{k}}(S_1^{{k}},\Theta_1^{{k}},V_2^{{1}},\dots,V_2^{{K}}), \nonumber \\ 
    \text{3. }
    &S_2^{{k}} = \hat{f}^{{k}}_{1}(S^{{k}}_1, W_1, V_{2}^{1},\dots,V_{2}^{K}, \Theta_1^{{k}}). \label{ap_c_lem_1_4}
\end{align}
Using the relations in \eqref{ap_c_lem_1_4}, we can define a function $\bar{f}_2^{{k}}: \mathscr{P}_1^{{k}} \times \mathscr{G}_1^{{k}} \times \mathcal{S}_1^{{k}} \times \mathcal{V}^{{1}} \times \dots \times \mathcal{V}^{{K}} \to \mathscr{P}_2^{{k}}$, such that
\begin{align}
    \Pi_2^{{k}} = \bar{f}_2^{{k}}(\Pi_1^{{k}},\Theta_1^{{k}},S_1^{{k}},V_2^{{1:K}}). \label{ap_c_lem_1_5}
\end{align}

Then, we substitute the relations \eqref{ap_c_lem_1_3} and \eqref{ap_c_lem_1_5} in the second term of \eqref{ap_d_to_go}, and define a new function $\underbar{c}_1^{{k}}:\mathscr{P}_1^{{k}} \times \dots \times \mathscr{P}_1^{{K}} \times \mathcal{S}_1^{{k}} \times \dots \times \mathcal{S}_1^{{K}} \times \mathscr{G}_1^{{k}} \times \dots \times \mathscr{G}_1^{{K}} \times \mathcal{V}^{{k}} \times \dots \times \mathcal{V}^{{K}} \times \mathcal{W} \to \mathbb{R}$ as follows,
\begin{gather}
    c_2^{{k}}(S_2^{{k}},\Theta_2^{{k}})
    =:\underbar{c}_1^{{k}}(\Pi_1^{{k:K}}, S_1^{k:K}, \Theta_1^{{k:K}}, V_2^{1:K}, W_1).
    \label{ap_c_lem_1_6}
\end{gather}

\textit{Step 4.3.} Using \eqref{ap_c_lem_1_6}, we can rewrite the cost-to-go as follows,
\begin{multline}
    \mathcal{J}^k_1 = \Tilde{c}_1^{{k}}(\pi_1^{{k}},\theta_1^{{k}})  \\  + \mathbb{E}^{\boldsymbol{\psi}^{{k}}}\Big[\underbar{c}_1^{{k}}(\Pi_1^{{k:K}}, S_1^{k:K}, \Theta_1^{{k:K}}, V_2^{1:K}, W_1) \; | \; a_1^k, \theta_{0}^k\Big]. \label{ap_d_to_go_2}
\end{multline}
Note that we can add the realizations of the information states $\pi^{k:K}_1$ and complete prescription $\theta_1^k$ to the conditioning in \eqref{ap_d_to_go_2} because, given the accessible information $a_t^k$, the prescription strategy $\boldsymbol{\psi}^{{k}}$, and the positional relationships $e^{[i,k]}: \Psi^{k} \mapsto \Psi^{i}$, $i \in \mathcal{B}^k$, agent $k$ can derive $\pi^{k:K}_1$ and $\theta_1^k$. This implies that,
\begin{multline}
    \mathcal{J}^k_1 = \Tilde{c}_1^{{k}}(\pi_1^{{k}},\theta_1^{{k}}) +  \mathbb{E}^{\boldsymbol{\psi}^{{k}}}\Big[\underbar{c}_1^{{k}}(\Pi_1^{{k:K}}, S_1^{k:K}, \Theta_1^{{k:K}}, V_2^{1:K},   \\ W_1) \; | \; a_1^k, \pi_1^{k:K} \theta_{0:1}^k\Big]. \label{ap_d_to_go_3}
\end{multline}
Now, we can expand the conditional expectation for the second term in \eqref{ap_d_to_go_3} as follows,
\begin{multline}
    \mathbb{E}^{\boldsymbol{\psi}^{{k}}}\Big[\underbar{c}_1^{{k}}(\Pi_1^{{k:K}}, S_1^{k:K}, \Theta_1^{{k:K}}, V_2^{1:K}, W_1) \; | \; a_1^k, \pi_1^{k:K}, \theta_{0:1}^{k}\Big] \\
    = \sum_{s^{{k:K}}_1,v^{{1:K}}_2, w_1}\underbar{c}^{{k}}_1(\pi_1^{{k:K}},s_1^{{k:K}},\theta_1^{{k:K}}, v_2^{1:K}) \cdot \mathbb{P}(V_2^{1:K} = v_2^{1:K}) \\
    \cdot \mathbb{P}(W_1 = w_1) \cdot \mathbb{P}^{\boldsymbol{\psi}^{{k}}}\big(S^{{k:K}}_1 = s^{{k:K}}_1 |a_1^{{k}},\pi_1^{{k:K}},\theta_{0:1}^{{k}}\big). \label{ap_c_lem_1_6_1}
    \end{multline}
From Assumption 2, the primitive random variables $\{V_2^{{1}},\dots,V_2^{{K}}, W_1\}$ have known probability distributions and are mutually independent. 
For the first term in \eqref{ap_c_lem_1_6_1}, we use \eqref{eq_ap_d_induction} to expand the realization $\theta_1^{{j}}$ of the complete prescription $\Theta_1^{{j}}$ of every agent ${{j}} \in \mathcal{B}^{{k+1}}$ as follows,
\begin{align} \label{ap_c_lem_1_7}
    \theta_1^{{j}} = \big(&\psi_1^{*[{{j}},{{1}}]}(\pi_1^{{{{j}}:K}}),\dots,\psi_1^{*[{{j}},{{j}}]}(\pi_1^{{{{j}}:K}}), \nonumber \\
    &\psi_1^{*[{{j}},{{j+1}}]}(\pi_1^{{{{j+1}}:K}}),\dots,\psi_1^{*[{{j}},{{K}}]}(\pi_1^{{K}})\big).
\end{align}
For the third term in \eqref{ap_c_lem_1_6_1}, we apply the result of Lemma \ref{lem_s_relation} for agents $\{k+1,\dots,K\}$ and simplify as follows,
\begin{align}
    \mathbb{P}^{\boldsymbol{\psi}^{{k}}}\big(&S^{{k:K}}_1 = s^{{k:K}}_1 \; \big|\; a_1^{{k}},\pi_1^{{k:K}},\theta_{0:1}^{{k}}\big) \nonumber \\
    = \mathbb{P}^{\boldsymbol{\psi}^{{k}}}\big(&s^{{k}}_1, a_1^{{k}} \setminus a_1^{{k+1}},\dots,  a_1^{{k}} \setminus a_1^{{K}}  \; \big| \; a_1^{{k}},\pi_1^{{k:K}},\theta_{0:1}^{{k}}\big) \nonumber \\
    = \mathbb{P}^{\boldsymbol{\psi}^{{k}}}\big(&S^{{k}}_1=s^{{k}}_1 \; \big| \; a_1^k, \pi_1^{{k:K}},\theta_{0:1}^{{k}}\big) = \pi_1^{{k}}(s_1^{{k}}), \label{ap_c_lem_1_8}
\end{align}
for any realization $(a_1^{{k}},\pi_1^{{k:K}}, \theta_{0:1}^{{k}})$ of positive probability. Note that in \eqref{ap_c_lem_1_8} we dropped the terms $\theta^{{k}}_1$ and $\pi_1^{k:K}$ from the conditioning because they are functions of the accessible information $a_t^k$, prescription strategy $\boldsymbol{\psi}^k$, and positional relationships $e^{[i,k]}: \Psi^k \mapsto \Psi^i$, $i \in \mathcal{B}^k$.
After substituting the results \eqref{ap_c_lem_1_7} and \eqref{ap_c_lem_1_8} in \eqref{ap_c_lem_1_6_1}, we can define a new function $\bar{c}_1^{{k}}: \mathscr{P}_1^{{k}} \times \dots \times \mathscr{P}_1^{{K}} \times \mathscr{G}_1^{{k}} \to \mathbb{R}$ such that
\begin{multline} \label{ap_c_lem_1_9}
     \bar{c}_1^{{k}}(\pi_1^{{k:K}},\theta_1^{{k}}) := \mathbb{E}^{\boldsymbol{\psi}^{{k}}}\Big[\underbar{c}_1^{{k}}(\Pi_1^{{k:K}}, S_1^{k:K}, \Theta_1^{{k:K}}, V_2^{1:K}, W_1)
     \\
      | \; a_1^k, \pi_1^{k:K}, \theta_{0:1}^{k}\Big].
\end{multline}

\textit{Step 4.4.} Thus, the cost-to-go for agent $k$ at time $t$ is given by
$
    \mathcal{J}_1^k = \Tilde{c}_1^{{k}}(\pi_1^{{k}},\theta_1^{{k}}) + \bar{c}_1^{{k}}(\pi_1^{{k:K}},\theta_1^{{k}}).
$
Then, using a procedure similar to the proof of Lemma \ref{two_stage_lemma}, there exists a function ${\psi}_1^{*[{{k}},{{i}}]}: \mathscr{P}_t^{{k}} \times \dots \times \mathscr{P}_t^{{K}} \to \mathscr{G}_t^{[{{k}},{{i}}]}$ for every ${{i}} \not \in \mathcal{B}^{{k+1}}$ such that the optimal prescription is
\begin{gather}
    {\Gamma}_1^{*[{{k}},{{i}}]} = {\psi}_1^{*[{{k}},{{i}}]}(\Pi_1^{{k}},\dots,\Pi_1^{{K}}).
\end{gather}

\textit{Step 5.} By mathematical induction starting with agent $K$, the optimal prescription strategy $\boldsymbol{\psi}_1^{*{{k}}}$ for agent ${{k}} \in \mathcal{K}$ satisfies the structural result
\begin{align}
\Gamma_1^{*[{{k}},{{i}}]} =
\begin{cases}
{\psi}_1^{*[{{k}},{{i}}]}(\Pi^{{k}}_1,\dots,\Pi^{K}_1), \quad \text{if } {{i}} \not\in \mathcal{B}^{{k}}, \\
{\psi}_1^{*[{{k}},{{i}}]}(\Pi^{{i}}_1,\dots,\Pi^{K}_1), \quad \text{if } {{i}} \in \mathcal{B}^{{k}}.
\end{cases}
\end{align}

\color{black}

\bibliographystyle{ieeetr}

\bibliography{References_per_date}

\end{document}